\documentclass[a4paper,10pt]{amsart}

\usepackage{graphics}
\usepackage{cancel}
\usepackage{color}
\usepackage{xcolor}
\usepackage{comment}
\usepackage{amsmath,mathtools}
\usepackage{amsfonts}
\usepackage{mathrsfs}
\usepackage{color}
\usepackage{algorithm}
\usepackage{algorithmic}
\usepackage{siunitx}
\usepackage{multirow}
\usepackage{stackrel}
\usepackage{booktabs}
\usepackage{tabularx}
\usepackage{arydshln}
\usepackage{cite}
\usepackage{url}
\usepackage{cleveref}
\usepackage{mdframed}

\usepackage{tikz}
\usepackage{pgfplots}
\pgfplotsset{compat=newest}
\usetikzlibrary{plotmarks}

\newcommand{\blue}{\color{blue}}
\newcommand{\red}{\color{red}}

\newtheorem{Theorem}{Theorem}[section]

\newtheorem{Corollary}{Corollary}[section]

\newtheorem{Lemma}{Lemma}[section]

\newtheorem{Remark}{Remark}[section]

\renewcommand{\baselinestretch}{1}

\newcommand{\R}{\mathbb{R}}
\renewcommand{\P}{\mathcal{P}}
\DeclareMathOperator{\sgn}{sgn}
\usepackage[top=2.5cm,bottom=2.5cm,left=2.5cm,right=2.5cm]{geometry}
\usepackage{lineno}

\definecolor{color1}{RGB}{216,27,96}
\definecolor{color2}{RGB}{30,136,229}
\definecolor{color3}{RGB}{255,193,7}

\begin{document}
\title[Block preconditioners for double saddle-point systems]{Spectral analysis of block preconditioners \\ for double saddle-point linear systems \\ with application to PDE-constrained optimization}
\author{Luca Bergamaschi  \and
\'Angeles Mart\'inez \and John W. Pearson \and Andreas Potschka}
\date{}
\footnotetext[1]{Department of Civil Environmental and Architectural Engineering, University of Padua, Via Marzolo, 9, 35100 Padua, Italy,
\texttt{E-mail}: luca.bergamaschi@unipd.it}
\footnotetext[2]{Department of Mathematics, Informatics and Geosciences, University of Trieste, Via Weiss, 2, 34128 Trieste, Italy,  
\texttt{E-mail}: amartinez@units.it}
\footnotetext[3]{School of Mathematics, The University of Edinburgh, James Clerk Maxwell Building, The King's Buildings, Peter Guthrie Tait Road, Edinburgh, EH9 3FD, United Kingdom,
\texttt{E-mail}: j.pearson@ed.ac.uk}
\footnotetext[4]{Institute of Mathematics, Clausthal University of Technology, Erzstr.~1, 38678 Clausthal-Zellerfeld, Germany,
\texttt{E-mail}: andreas.potschka@tu-clausthal.de}
\begin{abstract}
In this paper, we describe and analyze the spectral properties of a symmetric positive definite
	inexact block preconditioner 
	for a class of symmetric, double saddle-point linear systems.
	We develop a spectral analysis of the preconditioned matrix, 
	showing that its eigenvalues can be described in terms of the roots of a cubic polynomial with real coefficients.
	We illustrate the efficiency of the proposed preconditioners, and verify the theoretical bounds, in solving large-scale 
	PDE-constrained optimization problems.
\end{abstract}
\maketitle

\medskip


\smallskip


\section{Introduction}

Given positive integer dimensions $n \ge m \ge p$, consider the $(n+m+p) \times (n+m+p)$ double saddle-point linear system
of the form
\begin{equation}\label{Eq1}
        \mathcal{A}w = b, \qquad \text{where} \quad \mathcal{A}  =
        \begin{bmatrix}
        A & B^{\top} & 0\\
        B & 0 & C^{\top}\\
0 & C & E\end{bmatrix},
\end{equation}
with $A\in \mathbb{R}^{n \times n}$ is a symmetric positive definite (SPD) matrix, $B \in \mathbb{R}^{m \times n}$, and $C \in \mathbb{R}^{p \times m}$ having full row rank, and $E \in \mathbb{R}^{p \times p}$ a square positive semidefinite matrix.
Moreover $b$ and $w$ are vectors of length $n+m+p$.
This paper is concerned with an SPD block preconditioner for the numerical solution of (\ref{Eq1}).

Such linear systems arise in a number of scientific applications including constrained least squares problems \cite{Yuan}, constrained quadratic programming \cite{Han}, and magma--mantle dynamics \cite{Rhebergen}, to mention a few; see, e.g., \cite{Chen, Cai}. Similar block structures
arise e.g., in liquid crystal director modeling or in the coupled Stokes--Darcy problem, and the preconditioning of such linear systems has been considered in \cite{ChenRen, Szyld, Benzi2018, BeikBenzi2022}.
We also mention that block diagonal preconditioners for  problem (\ref{Eq1}) have been thoroughly studied in 
\cite{Bradley,SZ} and inexact block triangular preconditioners have been analyzed in \cite{Balani-et-al-2023a, Balani-et-al-2023b}.

Arguably the most prominent Krylov subspace methods for solving~\eqref{Eq1} are preconditioned variants of MINRES~\cite{minres} and GMRES~\cite{saad1986gmres}. In contrast to GMRES, the previously-discovered MINRES algorithm can explicitly exploit the symmetry of $\mathcal{A}$. As a consequence, MINRES features a three-term recurrence relation, which is beneficial for its implementation (low memory requirements because subspace bases need not be stored) and its purely eigenvalue-based convergence analysis (via the famous connection to orthogonal polynomials, see~\cite{fischer1996polynomial,greenbaum1997iterative}). Specifically, if the eigenvalues of the preconditioned matrix are contained within $[\rho^-_l, \rho^-_u] \cup [\rho^+_l, \rho^+_u]$, for $\rho^-_l < \rho^-_u < 0 < \rho^+_l < \rho^+_u$, then at iteration $k$ the Euclidean norm of the preconditioned residual $r_k$ satisfies the bound
\[
	\frac{\lVert r_k \rVert}{\lVert r_0 \rVert} \le 2 \left( \frac{\sqrt{\lvert \rho^-_l \rho^+_u \rvert} - \sqrt{\lvert \rho^-_u \rho^+_l \rvert}}{\sqrt{\lvert \rho^-_l \rho^+_u \rvert} + \sqrt{\lvert \rho^-_u \rho^+_l \rvert}} \right)^{\lfloor k / 2 \rfloor}.
\]
By contrast, GMRES needs to store subspace bases and its convergence analysis is in general dependent on the corresponding eigenspaces as well, which are more complicated to analyze than eigenvalues (see, e.g., \cite{embree2022descriptive}).

This motivates us to study a recently-proposed SPD block preconditioner~\cite{pearson2023symmetric} for~\eqref{Eq1}, which can be applied within MINRES. So far, tight eigenvalue bounds for inexact application of the preconditioner in all three diagonal blocks are missing. We close this gap by extending techniques from~\cite{Balani-et-al-2023b} together with an optimization-based paradigm to bound extremal roots of parameter-dependent polynomials.


The paper is structured as follows:
We present the ideal and approximate SPD double Schur complement preconditioner in Sec.~\ref{sec:ev_analysis} and analyze the case with vanishing $E$. In Sec.~\ref{sec:E_not_zero}, we extend the analysis to the case $E \neq 0$.
If the off-diagonal block $C$ is invertible, we can further refine the eigenvalue bounds via the working presented in Sec.~\ref{sec:C_invertible}. We illustrate and discuss the quality of the new bounds on numerical benchmark problems from the optimal control of partial differential equations (PDEs), with distributed and boundary observation, in Sec.~\ref{sec:numerics}. The paper ends with concluding remarks in Sec.~\ref{sec:conclusion}.

\section{Eigenvalue analysis of an inexact SPD preconditioner} \label{sec:ev_analysis}
We define \[S = BA^{-1}B^{\top}, \qquad X = E + C S^{-1} C^\top,\] and we consider  the following approximations in view
of a practical application of the preconditioner:
\begin{align}
	\label{hattilde}
	\widehat{A} ={}& A, && \nonumber \\
	\widetilde{S} ={}&  B \widehat A^{-1} B^\top, &\widehat{S}  \approx{}& \widetilde{S}, \\
	\widetilde{X} ={}& E + C \widehat S^{-1} C^\top, &\widehat{X}  \approx{}& \widetilde{X}\nonumber .
\end{align}
Here, $\widehat{A}$ is an SPD approximation of $A$ while $\widehat{S}$  and $\widehat X$
are SPD approximations of the exact Schur complements obtained from the approximations of $S$ and $X$, respectively.

We now consider the SPD preconditioner proposed in  \cite{pearson2023symmetric}
in the framework of multiple saddle-point linear systems. This is defined as 
$\P = \P_L \P_D^{-1} \P_L^\top$, where
\begin{equation*}
	{\mathcal{P}_L}=\begin{bmatrix} \widehat A &0&0\\B&-\widehat S &0\\0&C& \widehat X\end{bmatrix},
		\qquad
	{\mathcal{P}_D}=\begin{bmatrix} \widehat A &0&0\\0&\widehat S &0\\0&0& \widehat X\end{bmatrix}.
\end{equation*}
In this section we analyze the eigenvalue distribution  of the preconditioned matrix 
$\mathcal{P}^{-1}\mathcal{A}$ and relate its spectral properties 
with the extremal eigenvalues of 
$\widehat A^{-1} A$, $\widehat S^{-1} \widetilde{S}$, and $\widehat X^{-1} \widetilde{X}$, as defined in (\ref{hattilde}).

Finding the eigenvalues of ${\mathcal{P}}^{-1}\mathcal{A}$ is equivalent to solving
$$\P_D^{-1/2} \mathcal A \P_D^{-1/2} v = 
\lambda \P_D^{-1/2} P_L \P_D^{-1/2} \P_D^{-1/2} P_L^\top \P_D^{-1/2}  v, \qquad v = \begin{bmatrix} x\\y\\z\end{bmatrix},$$
	where $x$, $y$, and $z$ denote vectors of length $n,m$ and $p$, respectively.
Exploiting the block components of this generalized eigenvalue problem, we obtain
\begin{equation*}
	\begin{bmatrix}\bar{A}  &  R^\top  &  0\\
		R  &  0  &  K^\top\\
	0  &  K  & \bar{E} \end{bmatrix}
		\begin{bmatrix}x\\y\\z  \end{bmatrix}=\lambda 
			\begin{bmatrix}I  &    0  &  0\\R  &  -I  &  0\\0  &  K  &  I\end{bmatrix}
			\begin{bmatrix}I  &  R^\top  &  0\\0  &  -I  &  K^\top\\0  &  0  &  I\end{bmatrix}
\begin{bmatrix}x\\y\\z  \end{bmatrix},
\end{equation*}
which can be also written as 
\begin{equation}\label{Eq31.1}
	\begin{bmatrix}\bar{A}  &  R^\top  &  0\\
		R  &  0  &  K^\top\\
	0  &  K  & \bar{E} \end{bmatrix}
		\begin{bmatrix}x\\y\\z  \end{bmatrix}=\lambda 
			\begin{bmatrix}I  &    R^\top  &  0\\R  &  I + R R^\top  &  -K^\top\\0  &  -K  &  I + K K^\top\end{bmatrix}
\begin{bmatrix}x\\y\\z  \end{bmatrix},
\end{equation}
where $\bar{A}=\widehat A^{-\frac{1}{2}}A\widehat A^{-\frac{1}{2}} \equiv A_{\text{prec}},\: R =\widehat S^{-\frac{1}{2}}B\widehat A^{-\frac{1}{2}}$, $K = \widehat X^{-\frac{1}{2}}C\widehat S^{-\frac{1}{2}}$, and $\bar{E}=\widehat{X}^{-1/2}E\widehat{X}^{-1/2}$. 
Notice that 
\begin{align*} 
	R R^\top ={}&  \widehat S^{-\frac{1}{2}} \widetilde S \widehat S^{-\frac{1}{2}} 
	\equiv S_{\text{prec}}, \\
	K K^\top ={}&  \widehat X^{-\frac{1}{2}}  \left(\widetilde X -E\right) \widehat X^{-\frac{1}{2}} = \widehat X^{-\frac{1}{2}} \widetilde X \widehat X^{-\frac{1}{2}} -\bar{E} \equiv X_{\text{prec}} - \bar{E}.
\end{align*}
We define the Rayleigh quotient for a given symmetric matrix $H$ and nonzero vector $w$ as
\[ q(H,w) = \frac{w^{\top} H w}{w^{\top} w}.\]
The following indicators are used:
        \begin{align} 
		\gamma_{\min}^A  & \equiv \lambda_{\min} (\widehat A^{-1} A),  & \gamma_{\max}^A  & \equiv \lambda_{\max} (\widehat A^{-1} A), & \gamma_A(w_n) &= q(A_{\text{prec}}, w_n) \in [ \gamma_{\min}^A,  \gamma_{\max}^A ] \equiv I_A,  \nonumber \\[.6em]
		 \gamma_{\min}^{R}  & \equiv \lambda_{\min} (\widehat S^{-1} \widetilde{S}),  & \gamma_{\max}^{R}  & \equiv \lambda_{\max} (\widehat S^{-1} \widetilde{S}), &
		\gamma_{R}(w_m) &= q(S_{\text{prec}}, w_m)
                \in [ \gamma_{\min}^{R},  \gamma_{\max}^{R} ]\equiv I_R,  \nonumber \\[.6em]
		 \gamma_{\min}^{X}  & \equiv \lambda_{\min} (\widehat X^{-1} \widetilde{X}),  & \gamma_{\max}^{X} & \equiv \lambda_{\max} (\widehat X^{-1} \widetilde{X}), &
		\gamma_{X}(w_p) &= q(X_{\text{prec}}, w_p)
		\in [ \gamma_{\min}^{X},  \gamma_{\max}^{X} ] \equiv I_X,   \label{indA} \\[.6em]
		 \gamma_{\min}^{E}  & \equiv \lambda_{\min} (\widehat X^{-1} {E}),  & \gamma_{\max}^{E} & \equiv \lambda_{\max} (\widehat X^{-1} {E}), & \gamma_{E}(w_p)  &=
		q(\bar{E}, w_p)
                \in [ \gamma_{\min}^{E},  \gamma_{\max}^{E} ] \equiv I_E,   \nonumber \\[.6em]
   \label{indK}
		\gamma_{\min}^K  & \equiv \lambda_{\min} (K K^{\top}),  & \gamma_{\max}^K  & \equiv \lambda_{\max} (K K^{\top}) &
		\gamma_K(w_p)  &= q(K K^{\top},w_p) \in [ \gamma_{\min}^K,  \gamma_{\max}^K ]\equiv I_K. \nonumber
\end{align}
From the previous relations it easily holds that
$\gamma_X(w_p) = \gamma_K(w_p) + \gamma_E(w_p)$.

In the following, to make the notation easier we remove the argument $w_*$ whenever one of the indicators $\gamma_A$, $\gamma_R$, $\gamma_X$, $\gamma_E$, or $\gamma_K$ is used.
We finally make the following assumptions:
\begin{equation}
	\label{spect_int}
	\gamma_{\min}^A < 1 < \gamma_{\max}^A , \qquad 1 \in I_R, \qquad 1 \in I_X. 
\end{equation}


\subsection{Spectral analysis in a simplified case}
{We initially focus on the case $E \equiv 0$.
In this situation, $K K^\top = X_{\text{prec}}$, $\gamma_E \equiv 0$, and, consequently,
$\gamma_X = \gamma_K$.
Note that, in this case, we have also that $1 \in I_K$.
}

The eigenvalue problem \eqref{Eq31.1} then reads
\begin{align}
\bar{A}x-\lambda x ={}& (\lambda -1) R^\top y, \nonumber\\
(1-\lambda) Rx- \lambda(I + R R^\top) y  ={}& - (1+\lambda )K^\top z, \label{Eq33.1}\\
(1+ \lambda) Ky - \lambda(I + K K^\top) z ={}& 0. \label{Eq34.1}
\end{align}
Before stating the main results of this section, 
we premise a technical lemma, which generalizes \cite[Lemma 3.1]{BergaNLAA10}.
\begin{Lemma}\label{Le1}
	Let $Z$ be a symmetric matrix valued function  defined in  $F \subset \R$, and  
	\[0 \notin  [\min\{\sigma(Z(\zeta))\},\max\{\sigma(Z(\zeta))\}],\ \forall \zeta \in F,\]
  where $\sigma(Z(\zeta))$ denotes the spectrum of $Z(\zeta)$.
	Then, for arbitrary $s \neq 0$, there exists a vector $v \neq 0$ such that
\begin{equation*}
	\frac{s^\top(Z(\zeta))^{-1}s}{s^\top s}=\frac{1}{\gamma_Z},
	\qquad \text{with} \quad \gamma_Z = \frac{v^\top Z(\zeta) v}{v^\top v}.
\end{equation*}
\end{Lemma}
\begin{proof}
	For every $\zeta \in F$, either $Z(\zeta)$ or $-Z(\zeta)$  is SPD. In the first case we write
	$Z(\zeta) = L L^\top$ and hence
	\[ \frac{s^\top(Z(\zeta))^{-1}s}{s^\top s} = 
	 \frac{s^\top  L^{-\top} L^{-1} s}{s^\top s} \overset{v = L^{-1} s}{=}
	 \frac{v^\top v}{v^\top L L^\top v} =  
	 \frac{v^\top v}{v^\top Z(\zeta)  v} =  
	 \frac{1}{\gamma_Z}. \]
	 If $-Z(\zeta)$ is SPD,  the same result holds by applying the previous developments to $-Z(\zeta)$.
\end{proof}

We first concentrate on a classical saddle-point linear system. Letting
	\[\mathcal{A}_0 = 
\begin{bmatrix}
	A & B^{\top} \\
	B & 0\\
\end{bmatrix}, \qquad \mathcal{P}_0 = 
        \begin{bmatrix} \widehat A &0\\B&-\widehat S \end{bmatrix}
		\begin{bmatrix} \widehat A &0\\0&\widehat S \end{bmatrix}^{-1}
        \begin{bmatrix} \widehat A &B^\top\\0&-\widehat S \end{bmatrix},
		\]
		then the eigenvalues of 
$\mathcal{P}_0^{-1}\mathcal{A}_0$ are the same as those of
\begin{equation}\label{Eq31.0}
        \begin{bmatrix}\bar{A}  &  R^\top  \\
                R  &  0  \\ \end{bmatrix}
                \begin{bmatrix}x\\y  \end{bmatrix}=\lambda
                        \begin{bmatrix}I  &    R^\top  \\R  &  I + R R^\top \\ \end{bmatrix}
\begin{bmatrix}x\\y \end{bmatrix}. 
\end{equation}

			The following theorem characterizes the eigenvalues of the preconditioned matrix $\mathcal{A}_0$,
			and hence a classical saddle-point linear system preconditioned by ${\mathcal{P}_0}$, in terms
			of the indicators $\gamma_A$ and $\gamma_R$.
The findings of this theorem also constitute the basis for the proof of \Cref{Theo:pi}.
\begin{Theorem} \label{Theo:p}
	\label{saddle-point}
	The eigenvalues of $\mathcal{P}_0^{-1}\mathcal{A}_0$  are either contained in $[\gamma_{\min}^A, \gamma_{\max}^A]$, or they
	are the roots of the $(\gamma_A,\gamma_R)$-parametric family of polynomials
	\begin{equation} \label{plambda} 
	p(\lambda; \gamma_A, \gamma_R) = \lambda^2 - \lambda(\gamma_A \gamma_R + \gamma_A - 2\gamma_R)  - \gamma_R \qquad \text{for }~\gamma_A \in I_A,~\gamma_R \in I_R.
\end{equation}
\begin{Remark}[Notation]
  For quantities which depend on $\gamma$-indicators $\gamma_A$, $\gamma_R$, $\gamma_X$, $\gamma_E$, or $\gamma_K$, we will frequently suppress the $\gamma$-indicator arguments unless they are substituted by special values. Such quantities can be polynomials or their roots. For instance, we will frequently write $p(\lambda)$ as shorthand for $p(\lambda; \gamma_A, \gamma_R)$ but may explicitly write, e.g., $p(\lambda; \gamma^A_{\min}, \gamma^R_{\max})$.
\end{Remark}
\end{Theorem}
\begin{proof} (of Thm.~\ref{Theo:p})
	Assume that $\lambda \not \in [\gamma^A_{\min}, \gamma^A_{\max}] $.
	Then from the first row of \eqref{Eq31.0} we obtain
\begin{equation}\label{Eq65}
x = (1-\lambda) (\lambda I -\bar{A})^{-1} R^\top y.
\end{equation}
	Inserting (\ref{Eq65}) into the second row of \eqref{Eq31.0} yields
\begin{equation*}
	\underbrace{\left( (\lambda-1)^2 R (\lambda I -  \bar{A})^{-1} R^\top - \lambda (I + R R^\top) \right)}_{\text{\normalsize$Y(\lambda) $}} y = 0.
\end{equation*}
	Applying Lemma \ref{Le1} to $Z(\lambda) = \lambda I - \bar{A}$ and setting $u = R^\top y$ yields
	\begin{align}
		\nonumber 0 ={}& 	\frac{y^\top Y(\lambda) y}{y^\top y} =
		(\lambda-1)^2 \frac{y^\top R (\lambda I -  \bar{A})^{-1} R^\top y}{y^\top y} - \lambda \left( 1 + \frac{y^\top R R^\top y}
		{y^\top y} \right) \\
		\nonumber ={}& (\lambda-1)^2 \frac{u^\top (\lambda I -  \bar{A})^{-1} u}{u^\top u}\frac{y^\top R R^\top y}{y^\top y} - \lambda \left( 1 + \frac{y^\top R R^\top y}
		{y^\top y} \right) \\
		\label{yZy} ={}& \frac{(\lambda-1)^2}{\lambda - \gamma_A} \gamma_R - \lambda(1 + \gamma_R) 
	= \frac{-\lambda^2 + \lambda(\gamma_A \gamma_R + \gamma_A - 2 \gamma_R) + \gamma_R}{\lambda - \gamma_A}\equiv 
	\frac{p(\lambda)}{\gamma_A - \lambda},
    \end{align}
	with
	\[p(\lambda) = \lambda^2 - \lambda(\gamma_A \gamma_R + \gamma_A - 2\gamma_R)  - \gamma_R. \qedhere \]
\end{proof}
If $\gamma_A = 1$ then the roots of $p(\lambda)$ are $-\gamma_R , 1$. 
To bound
the roots of $p(\lambda)$ when $\gamma_A \ne 1$, we develop a general result for the extremal roots of a polynomial with the help of an optimization problem. We shall use the result repeatedly in the remainder of the paper. 

\begin{Lemma} \label{lem:extremal_roots}
  Let $\boldsymbol{\gamma} \in \R^d$, $q(\lambda; \boldsymbol{\gamma})$ be a $\boldsymbol{\gamma}$-dependent polynomial in $\lambda$, and consider the set
  \[
    \mathbb{X} = \{ (\lambda, \boldsymbol{\gamma}) \in \R^{d+1} \mid q(\lambda; \boldsymbol{\gamma}) = 0 
    \text{ and } \gamma_j \in [\gamma^j_{\min}, \gamma^j_{\max}] \text{ for } j = 1, \dotsc, k \}.
  \]
  If for $(\lambda, \boldsymbol{\gamma}) \in \mathbb{X}$  the root $\lambda$ of $q$ is locally extremal over $\mathbb{X}$ and satisfies $\frac{\partial q}{\partial \lambda}(\lambda; \boldsymbol{\gamma}) \neq 0$, then exactly one of the following three cases holds:
  \begin{alignat*}{3}
    \emph{(a)} &~& \delta \frac{\partial q}{\partial \gamma_j}(\lambda; \boldsymbol{\gamma}) &
    \ge 0 & \text{and } \gamma_j &= \gamma^j_{\min},\\
    \emph{(b)} & & \delta \frac{\partial q}{\partial \gamma_j}(\lambda; \boldsymbol{\gamma}) &
    \le 0 & \text{and } \gamma_j &= \gamma^j_{\max},\\
    \emph{(c)} & & \frac{\partial q}{\partial \gamma_j}(\lambda; \boldsymbol{\gamma}) &
    = 0 & \text{ and } \gamma_j &\in (\gamma^j_{\min}, \gamma^j_{\max}),
  \end{alignat*}
  where the sign $\delta \in \{\pm 1\}$ is defined by
  \[
    \delta =
    \begin{cases}
      -\sgn \frac{\partial q}{\partial \lambda}(\lambda; \boldsymbol{\gamma}) & \text{if $\lambda$ is a local minimum},\\
      +\sgn \frac{\partial q}{\partial \lambda}(\lambda; \boldsymbol{\gamma}) & \text{if $\lambda$ is a local maximum}.\\
    \end{cases}
  \]
\end{Lemma}
\begin{proof}
  We set $\bar{\delta} = 1$ if $\lambda$ is locally minimal or $\bar{\delta} = -1$ if $\lambda$ is locally maximal, and consider the optimization problem:
  \begin{equation*}
    \min ~~ \bar{\delta} \lambda \qquad \text{s.t.} \quad (\delta, \boldsymbol{\gamma}) \in \mathbb{X}.
  \end{equation*}
  As $(\lambda, \boldsymbol{\gamma}) \in \mathbb{X}$ satisfies $\frac{\partial q}{\partial \lambda}(\lambda; \boldsymbol{\gamma}) \neq 0$, the Linear Independence Constraint Qualification holds in $(\lambda, \boldsymbol{\gamma})$. By the Karush--Kuhn--Tucker necessary optimality conditions, there then exists a Lagrange multiplier $\psi \in \R$ such that for the Lagrangian
  \[
    L(\lambda, \boldsymbol{\gamma}, \psi) = \bar{\delta} \lambda + \psi q(\lambda; \boldsymbol{\gamma}),
  \]
  it necessarily holds that
  \begin{align}
    \label{eqn:dLdlambda}
    \frac{\partial L}{\partial \lambda}(\lambda, \boldsymbol{\gamma}, \psi) &= \bar{\delta} + \psi \frac{\partial q}{\partial \lambda}(\lambda; \boldsymbol{\gamma}) = 0,\\
    \label{eqn:dLdgamma}
    \frac{\partial L}{\partial \gamma_j}(\lambda, \boldsymbol{\gamma}, \psi) &= \psi \frac{\partial q}{\partial \gamma_j}(\lambda; \boldsymbol{\gamma})
    \begin{cases}
      \ge 0 & \text{if } \gamma_j = \gamma^j_{\min}, \\
      \le 0 & \text{if } \gamma_j = \gamma^j_{\max}, \\
      = 0 & \text{if } \gamma_j \in (\gamma^j_{\min}, \gamma^j_{\max}),
    \end{cases}
  \end{align}
  for $j = 1, \dotsc, d$. Resolving~\eqref{eqn:dLdlambda} for $\psi$, we have that the assertion follows from~\eqref{eqn:dLdgamma} with
  \[
    \delta := \sgn \psi = -\sgn \bar{\delta} \sgn \frac{\partial q}{\partial \lambda}(\lambda; \boldsymbol{\gamma}). \qedhere
  \]
\end{proof}
We can finally bound the roots of $p(\lambda)$ when $\gamma_A \neq 1$.
\begin{Lemma}
	\label{Lem:optp}
  Let $\gamma_A \ne 1$, $\eta(\gamma_A, \gamma_R) = \frac{1}{2} (\gamma_R + 1) \gamma_A - \gamma_R$, and denote
	with $\lambda_-(\gamma_A, \gamma_R)$ and $\lambda_+(\gamma_A, \gamma_R)$ the two roots of $p(\lambda)$, that is
	\[ \lambda_\pm(\gamma_A, \gamma_R) = \eta(\gamma_A, \gamma_R) \pm \sqrt{\eta(\gamma_A, \gamma_R)^2 + \gamma_R}. \]
	Then,
	\begin{equation}
		\lambda \in \left[\lambda_-(\gamma_{\min}^A, \gamma_{\max}^R) , \lambda_-(\gamma_{\max}^A, \gamma_{\min}^R)\right]
		\cup \left[
			\lambda_+(\gamma_{\min}^A, \gamma_{\min}^R),  \lambda_+(\gamma_{\max}^A, \gamma_{\max}^R) \right]. 
		\label{pbounds}
	\end{equation}
\end{Lemma}
\begin{proof}
	We first observe that $p(\lambda) = \lambda^2 - 2 \eta \lambda - \gamma_R$, which confirms the definition of $\lambda_{\pm}(\gamma_A, \gamma_R)$ and shows that $\lambda_- < 0 < \lambda_+$ and $\lambda_- < \eta < \lambda_+$. 
	{Before writing the 
	partial derivatives and applying Lemma~\ref{lem:extremal_roots} to $p$, we need to establish that
	\begin{align} \label{assertionA}  (2 - \gamma_A) \lambda_{+} - 1 & <  0, \\
		\label{assertionB}  (2 - \gamma_A) \lambda_{-} - 1 & < 0.
	\end{align}
	We first observe that
	\begin{equation}
		\label{partialR}
    p\left(\frac{1}{2 - \gamma_A}\right)
    = \frac{1}{(2 - \gamma_A)^2} - \frac{\gamma_R (\gamma_A - 2) + \gamma_A}{2 - \gamma_A} - \gamma_R
    = \left( \frac{\gamma_A - 1}{\gamma_A - 2} \right)^2 > 0 \qquad \text{for } \gamma_A \neq 1.
	\end{equation}
	Then, if $\gamma_A \ge 2$, \eqref{assertionA} is obviously true. If, conversely, $\gamma_A < 2$, then
	\eqref{partialR} implies
	that $\lambda_+ < 1 / (2 - \gamma_A)$, so that \eqref{assertionA} is proved. If $\gamma_A < 2$ then \eqref{assertionB} is true,
	if instead $\gamma_A \ge 2$, \eqref{partialR} implies that
	$- 1/(\gamma_A - 2) < \lambda_-$, which is equivalent to \eqref{assertionB}.}
  Based on the partial derivatives 
  \begin{align*}
    \frac{\partial p}{\partial \lambda}(\lambda) &= 2 (\lambda - \eta), &
    \frac{\partial p}{\partial \gamma_A}(\lambda) &= -(\gamma_R + 1) \lambda, &
    \frac{\partial p}{\partial \gamma_R}(\lambda) &= (2 - \gamma_A) \lambda - 1,
    \intertext{we can summarize}
    \frac{\partial p}{\partial \lambda}(\lambda_-) &< 0, & 
    \frac{\partial p}{\partial \gamma_A}(\lambda_-) &> 0, &
    \frac{\partial p}{\partial \gamma_R}(\lambda_-) &< 0, \\
    \frac{\partial p}{\partial \lambda}(\lambda_+) &> 0, &
    \frac{\partial p}{\partial \gamma_A}(\lambda_+) &< 0, &
    \frac{\partial p}{\partial \gamma_R}(\lambda_+) &< 0.    
  \end{align*}
  We can now apply Lemma~\ref{lem:extremal_roots} to $p$ for $\lambda_{\pm}$ and $\boldsymbol{\gamma} = (\gamma_A, \gamma_R)^\top$. If $\lambda_-$ is a local minimum, then
  \[
    \delta = -\sgn \frac{\partial p}{\partial \lambda}(\lambda_-) = 1,
  \]
  which implies $\gamma_A = \gamma^A_{\min}$ (only case (a) possible) and $\gamma_R = \gamma^R_{\max}$ (only case (b) possible). The same reasoning can be applied to all three remaining combinations of $\lambda_{\pm}$ being a local minimum/maximum, to show that~\eqref{pbounds} holds.
\end{proof}
\begin{Corollary}
	\label{Cor:p}
	       Any eigenvalue $\lambda$  of $\mathcal{P}_0^{-1}\mathcal{A}_0$   lies in
        $I_- \cup I_+$, where
        \begin{equation*}
		I_- = [\lambda_-(\gamma_{\min}^A, \gamma_{\max}^R), \lambda_-(\gamma_{\max}^A, \gamma_{\min}^R)], \qquad
		I_+ = [\gamma_{\min}^A, \lambda_+(\gamma_{\max}^A, \gamma_{\max}^R)].
	\end{equation*}
\end{Corollary}
\begin{proof}
	After observing that
	$-\gamma_R \in I_-$, $1 \in I_+$, and $\gamma_A \le \lambda_+(\gamma_A, \gamma_R)$, implying that
$	\gamma_{\min}^A \le \lambda_+(\gamma_{\min}^A, \gamma_{\min}^R) $,
	the statement follows from \Cref{Theo:p} and \Cref{Lem:optp}.
\end{proof}
We are now ready to characterize the eigenvalues of the preconditioned matrix 
	$\mathcal{P}^{-1} \mathcal{A}$. {To this end we require a further hypothesis on the 
	eigenvalues of the preconditioned $(1,1)$ block, in addition to \eqref{spect_int}, specifically that
	\[ \gamma_{\max} ^A < 2.\]
	}
\begin{Theorem}
	\label{Theo:pi}
	The eigenvalues of $\mathcal{P}^{-1} \mathcal{A}$ either belong to 
	$I_- \cup I_+$,  or they are solutions to the cubic polynomial equation
		\[\pi(\lambda; \gamma_A, \gamma_R, \gamma_K) \equiv (1+ \lambda)^2 (\gamma_A - \lambda) \gamma_K + p(\lambda; \gamma_A, \gamma_R) \lambda (1+\gamma_K)  = 0.  \]
\end{Theorem}

\begin{proof}
	Assuming $\lambda \not \in [\gamma^A_{\min}, \gamma^A_{\max}] $ and
	inserting (\ref{Eq65}) into \eqref{Eq33.1} yields
\begin{equation}
	\label{Zlambda}
	Y(\lambda) y = -(1+\lambda) K^\top z, \qquad \text{with} \quad 
	Y(\lambda) = (\lambda-1)^2 R (\lambda I -  \bar{A})^{-1} R^\top - \lambda (I + R R^\top).
\end{equation}
	Using \Cref{Theo:p}, we have that
	 if $\lambda \not \in I_- \cup I_+$  then $Y(\lambda) $ is either positive or negative definite,
	 and hence invertible.
	 Based on (\ref{Zlambda}), we can write
	 \[ y = -(1+\lambda) Y(\lambda)^{-1} K^\top z, \]
	 and substitution into (\ref{Eq34.1}) yields
	\begin{equation} \label{noE} 
	-\left((1+ \lambda)^2 K Y(\lambda)^{-1} K^\top + \lambda(I + K K^\top) \right) z =0.  \end{equation}
	Let us now pre-multiply~\eqref{noE} by  $\frac{z^\top}{z^\top z} $ to establish
	\[-\frac{z^\top \left((1+ \lambda)^2 K Y(\lambda)^{-1} K^\top \right) z}{z^\top z} - 
	\lambda\left(1 + \frac{z^\top K K^\top z}{z^\top z} \right)  =0.  \]
	Setting $s = K^\top z$, and multiplying numerator and denominator of the first term by $s^\top s$, we obtain
	\begin{equation}
		\label{penultimate}
	(1+\lambda)^2\frac{s^\top Y(\lambda)^{-1} s}{s^\top s} \frac{z^\top K K^\top z}{z^\top z} +
	\lambda\left(1 + \frac{z^\top K K^\top z}{z^\top z} \right)  =0.  
	\end{equation}
	Using (\ref{yZy}) and applying Lemma \ref{Le1} to $Y(\lambda)$, we have that
	\[ \frac{s^\top Y(\lambda)^{-1} s}{s^\top s} = \frac{\gamma_A - \lambda}{p(\lambda)}, \]
	which, substituted into (\ref{penultimate}), yields
	\[ (1+ \lambda)^2 \frac {\gamma_A - \lambda}{p(\lambda)} \gamma_K + \lambda (1+\gamma_K) = 0,\]
	the zeros of which outside $I_- \cup I_+$ characterize the eigenvalues of the preconditioned matrix,
	as well as the zeros of $\pi(\lambda)$.
\end{proof}
	\begin{Remark}
		Notice that the indicator $\gamma_A$ above, as well as $\gamma_A$, $\gamma_R$ in the definition of $p(\lambda)$, are not 
		exactly those of \eqref{plambda},   since the vectors by which the corresponding Rayleigh quotients are defined 
		are different. However, we indicate them with the same symbol as, in all cases, they satisfy the conditions defined in \eqref{indA}.
	\end{Remark}

We consider separately the case in which $\gamma_A = 1$. In this case
\begin{align*} \pi(\lambda; 1, \gamma_R, \gamma_K) ={}& (1+ \lambda)^2 (1 - \lambda) \gamma_K + (\lambda-1) (\lambda + \gamma_R)\lambda (1+\gamma_K)  \\
	={}& (\lambda - 1) \underbrace{\left(\lambda^2 + \lambda (\gamma_R(\gamma_K+1) -2 \gamma_K\right) - \gamma_K)}_{\text{\normalsize $p(-\lambda; \gamma_R, \gamma_K) = c(\lambda; \gamma_R, \gamma_K)$}}.
\end{align*}
This shows that $\lambda=1$ is a root of $\pi$, the remaining roots ($c_-, c_+$) being the two distinct solutions of $c(\lambda) = 0$.

Applying \Cref{Lem:optp} to $c(\lambda)$, we conclude that 
\begin{equation} \label{cbounds}
	\lambda \in \left[-\lambda_+(\gamma_{\max}^R, \gamma_{\max}^K), -\lambda_+(\gamma_{\min}^R, \gamma_{\min}^K)\right] \cup
	  \left[-\lambda_-(\gamma_{\max}^R, \gamma_{\min}^K), -\lambda_-(\gamma_{\min}^R, \gamma_{\max}^K)\right]. 
\end{equation}
It is also easy to show that $1$ belongs to the positive interval. First $-1 \in
\left[-\lambda_+(\gamma_{\max}^R, \gamma_{\max}^K), -\lambda_+(\gamma_{\min}^R, \gamma_{\min}^K)\right]$ since $c(-1) = 
(1 - \gamma_R) (1 + \gamma_K)$, showing that $c_- \le -1$ if $\gamma_R \ge  1$ and  $c_- \ge -1$ if $\gamma_R \le 1$. 
If $\gamma_K = \gamma_{\max}^K$ and $\gamma_R = \gamma_{\max}^R$,
from $c_- c_+ = -\gamma_{\max}^K \le -1$ and $c_- \ge -1$ it follows that $c_+ \ge 1$. Conversely, if 
$\gamma_K = \gamma_{\min}^K$ and $\gamma_R = \gamma_{\min}^R$,
from $c_- c_+ = -\gamma_{\min}^K \ge -1$ and $c_- \le -1$ it follows that $c_+ \le 1$. 

If instead $\gamma_A \ne 1$, $\pi(\lambda)$ satisfies (see also \Cref{newplot}):
\begin{equation}
	\label{piconditions}
  \begin{aligned}
	\lim _{\lambda \to -\infty} \pi(\lambda) & = -\infty, &
	\pi(\lambda_-) & = \gamma_K (1+\lambda_-)^2 (\gamma_A - \lambda_-) \ge 0, &
	\pi(0) & = \gamma_A \gamma_K  > 0, \\
	\lim _{\lambda \to +\infty} \pi(\lambda) & = +\infty, &
	\pi(\lambda_+) & = \gamma_K (1+\lambda_+)^2 (\gamma_A - \lambda_+) < 0, &
	\pi(\gamma_A) & = {-\gamma_A\gamma_R (1+\gamma_K)(\gamma_A-1)^2 } <  0,
  \end{aligned}
\end{equation}
so we conclude that $\pi(\lambda) = 0$ has  three distinct real roots
\[
  \mu_a(\gamma_A, \gamma_R, \gamma_K) < 0 < \mu_b(\gamma_A, \gamma_R, \gamma_K) < \mu_c(\gamma_A, \gamma_R, \gamma_K).
\]

\begin{Lemma}
	\label{Lem:optpi}
	If $\gamma_A \neq 1$, the roots  of $\pi(\lambda)$
	belong to $I^\pi_- \cup I^\pi_+$, where
	 \begin{align*} I^{\pi}_-  ={}& \left[\mu_a(\gamma_{\min}^A,  \gamma_{\max}^R,  \gamma_{\max}^K), 
		                         \mu_a(\gamma_{\max}^A,  \gamma_{\min}^R,  \gamma_{\min}^K) \right],   \nonumber \\
		 I ^\pi_+ ={}& \left[\mu_b(\gamma_{\min}^A,  \gamma_{\max}^R,  \gamma_{\min}^K), 
                \max\{
                        \mu_c(\gamma_{\max}^A,  \gamma_{\min}^R,  \gamma_{\max}^K) ,
                        \mu_c(\gamma_{\max}^A,  \gamma_{\max}^R,  \gamma_{\max}^K), \beta_c(\gamma_{\max}^A, \gamma_{\max}^K) \}
                        \right],
	 \end{align*}
			and 
			                \[ \beta_c(\gamma_A, \gamma_K) = \min\left\{\frac{1}{2-\gamma_A}, \gamma_K + \sqrt{(\gamma_K)^2 + \gamma_K}\right\}.\]

\end{Lemma}

\begin{figure}[h!]
  \hspace{-1cm}
\begin{minipage}{8.4cm}
\begin{center}
\begin{tabular} {|r||ccc|} \hline
  $\lambda =$ & $\mu_a$ & $\mu_b$ & $\mu_c$ \\ \hline \hline
  $\lambda$ & $-$ & $+$ & $+$ \\
  $p(\lambda) $ & $+$ & $-$ &  $+$ \\
  $\frac{\partial \pi}{\partial \lambda}(\lambda)$ & $+$ & $-$ &  $+$  \\
  $\gamma_A - \lambda$ &  $+$ & $+$ &  $-$ \\ \hline
\end{tabular}
\end{center}
  \caption{Summary of the signs of the relevant quantities for the proof of Lemma~\ref{Lem:optpi}.}
  \label{sign}
\end{minipage}
  \hspace{-3mm}
\begin{minipage}{7cm}
  \newcommand\GAMMAA{1.639}
  \newcommand\GAMMAR{0.734}
  \newcommand\GAMMAK{0.251}
  \begin{tikzpicture}
    \begin{axis}[xmin=-0.7, xmax=2.3, ymin=-1.5, ymax=2, grid=major,
        legend style={at={(0.5,0.97)}, anchor=north}, xlabel={$\lambda$}]
      \newcommand\MUA{-0.503254554435484}
      \newcommand\MUB{0.435434175555416}
      \newcommand\MUC{1.877337904880065}
      \newcommand\CA{2*\GAMMAK + \GAMMAA + \GAMMAR*(1+\GAMMAK)*(\GAMMAA-2)}
      \newcommand\CB{\GAMMAR*(1+\GAMMAK) + \GAMMAK*(1-2*\GAMMAA)}
      \newcommand\CC{\GAMMAA*\GAMMAK}
      \addplot[domain=-0.7:2.3, samples=100, color2, very thick]
      {((x - (\CA)) * x - (\CB)) * x + (\CC)};
      \addlegendentry{$\pi(\lambda)$}
      \addplot[domain=-0.7:2.3, samples=100, color1, very thick]
      {x^2 - ((\GAMMAA)*(1+\GAMMAR)-2*(\GAMMAR))*x - \GAMMAR};
      \addlegendentry{$p(\lambda)$}
      \addplot[mark=*, mark size=2pt, only marks] coordinates {(\GAMMAA, 0)}; \addlegendentry{$\gamma_A$}
      \addplot[mark=o, mark size=2pt, only marks, color2] coordinates {(\MUA, 0)}; \addlegendentry{$\mu_a$}
      \addplot[mark=+, mark size=3pt, only marks, color2] coordinates {(\MUB, 0)}; \addlegendentry{$\mu_b$}
      \addplot[mark=x, mark size=3pt, only marks, color2] coordinates {(\MUC, 0)}; \addlegendentry{$\mu_c$}
    \end{axis}
	\end{tikzpicture}
	\caption{Qualitative plots of $\pi(\lambda)$ and $p(\lambda)$ with $\gamma_A=\GAMMAA$, $\gamma_R=\GAMMAR$, and $\gamma_K=\GAMMAK$.}
  \label{newplot}
\end{minipage}
\end{figure}

\begin{proof}
  With the aim of applying Lemma~\ref{lem:extremal_roots} to $\pi(\lambda; \gamma_A, \gamma_R, \gamma_K)$, we determine the signs of the partial derivatives of $\pi$ within the three roots $\mu_a$, $\mu_b$, and $\mu_c$. For the signs of $\frac{\partial \pi}{\partial \lambda}$, we refer to Figure \ref{sign}.

	For $\frac{\partial \pi}{\partial \gamma_A}$, we first write an alternative expression for $p(\lambda$):
	 \begin{equation} \label{newplambda}
    p(\lambda) = \lambda (\gamma_R + 1) (\lambda - \gamma_A ) - \gamma_R (\lambda - 1)^2,
  \end{equation}
	and we collect the $\gamma_A$ terms in $\pi(\lambda)$ by rearranging:
  \begin{align*}
    \pi(\lambda) &= (1+ \lambda)^2 (\gamma_A - \lambda) \gamma_K +\lambda p(\lambda) (1 + \gamma_K) \\
    &= (\gamma_A - \lambda)\left( (1+ \lambda)^2 \gamma_K -  \lambda^2 (1 + \gamma_R) (1 + \gamma_K) \right)
    -\lambda \gamma_R (\lambda - 1)^2 (1+\gamma_K) \\
    &= (\gamma_A - \lambda) \frac{\partial \pi}{\partial \gamma_A}(\lambda) - \lambda \gamma_R (\lambda - 1)^2 (1+\gamma_K),
  \end{align*}
  which implies that
  \[
    \frac{\partial \pi}{\partial \gamma_A}(\lambda) = \frac{\pi(\lambda) + \lambda \gamma_R (\lambda - 1)^2 (1+\gamma_K)}{\gamma_A - \lambda}.
  \]
  With the aid of Figure \ref{sign}, we obtain for the zeros of $\pi(\lambda)$ that
  \begin{align*}
    \frac{\partial \pi}{\partial \gamma_A}(\mu_a) &< 0, &
    \frac{\partial \pi}{\partial \gamma_A}(\mu_b) &> 0, &
    \frac{\partial \pi}{\partial \gamma_A}(\mu_c) &< 0,
  \end{align*}
  which are the opposite signs of $\frac{\partial \pi}{\partial \lambda}(\lambda)$ for $\lambda = \mu_a, \mu_b, \mu_c$. Hence, Lemma~\ref{lem:extremal_roots} delivers that $\gamma_A = \gamma^A_{\min}$ for a local minimum of $\lambda$ and $\gamma_A = \gamma^A_{\max}$ for a local maximum.

  For $\frac{\partial \pi}{\partial \gamma_K}$, we rearrange
  \[
    \pi(\lambda) = \left( (1 + \lambda)^2 (\gamma_A - \lambda) + \lambda p(\lambda) \right) \gamma_K + \lambda p(\lambda),
  \]
  which implies (because $p(\lambda)$ is independent of $\gamma_K$) that
  \[
    \frac{\partial \pi}{\partial \gamma_K}(\lambda) = \frac{\pi(\lambda) - \lambda p(\lambda)}{\gamma_K}.
  \]
  From Figure \ref{sign}, we can deduce that
  \begin{align*}
    \frac{\partial \pi}{\partial \gamma_K}(\mu_a) &> 0, &
    \frac{\partial \pi}{\partial \gamma_K}(\mu_b) &> 0, &
    \frac{\partial \pi}{\partial \gamma_K}(\mu_c) &< 0,
  \end{align*}
  where only the last two partial derivatives have a sign opposing that of $\frac{\partial \pi}{\partial \lambda}$.
  Thus, Lemma~\ref{lem:extremal_roots} delivers that a local minimum of $\mu_a$ implies $\gamma_K = \gamma^K_{\max}$ and a local maximum of $\mu_a$ implies $\gamma_K = \gamma^K_{\min}$, while a local minimum of $\mu_b$ implies $\gamma_K = \gamma^K_{\min}$ and a local maximum of $\mu_c$ implies $\gamma_K = \gamma^K_{\max}$.

  For $\frac{\partial \pi}{\partial \gamma_R}$, we have that
  \[
    \frac{\partial \pi}{\partial \gamma_R}(\lambda) = (1 + \gamma_K) \lambda \frac{\partial p}{\partial \gamma_R}(\lambda), \quad \text{where } 
    \frac{\partial p}{\partial \gamma_R}(\lambda) = (2 - \gamma_A) \lambda - 1.
  \]
  Since $\mu_a < 0$, $2 - \gamma_A > 0$, and $\mu_b < \gamma_A$ we have that
  \begin{align*}
    \frac{\partial p}{\partial \gamma_R}(\mu_a) &< 0, &
    \frac{\partial p}{\partial \gamma_R}(\mu_b) &<
    \frac{\partial p}{\partial \gamma_R}(\gamma_A) = -(\gamma_A-1)^2 < 0,
    \intertext{implying}
    \frac{\partial \pi}{\partial \gamma_R}(\mu_a) &> 0, &
    \frac{\partial \pi}{\partial \gamma_R}(\mu_b) &< 0,
  \end{align*}
  both exhibiting the same signs as those of $\frac{\partial \pi}{\partial \lambda}$.
  Thus, we obtain that a local minimum of $\mu_a$ requires $\gamma_R = \gamma^R_{\max}$, a local maximum of $\mu_a$ requires $\gamma_R = \gamma^R_{\min}$, and a local minimum of $\mu_b$ requires $\gamma_R = \gamma^R_{\max}$ by Lemma~\ref{lem:extremal_roots}.

  For the third root $\mu_c > 0$ we distinguish three cases:
  First, if $\mu_c < \frac{1}{2-\gamma_A}$ then $\frac{\partial \pi}{\partial \gamma_R}(\mu_c)$ has the same (positive) sign as $\frac{\partial \pi}{\partial \lambda}(\mu_c)$ and Lemma~\ref{lem:extremal_roots} requires $\gamma_R = \gamma^R_{\min}$ for a local maximum of $\mu_c$.
  Second, if $\mu_c > \frac{1}{2-\gamma_A}$ then $\frac{\partial \pi}{\partial \gamma_R}(\mu_c)$ has the opposite sign as $\frac{\partial \pi}{\partial \lambda}(\mu_c)$ and Lemma~\ref{lem:extremal_roots} requires $\gamma_R = \gamma^R_{\max}$ for a local maximum of $\mu_c$.
  Third, if $\mu_c = \frac{1}{2-\gamma_A}$, the simple bound
  \[ \mu_c \le \frac{1}{2-\gamma_{\max}^A} \]
  follows. As this bound is not useful if $\gamma_{\max}^A$ is close to 2, a refinement is helpful: Exploiting the condition
  $\pi(\mu_c) = 0$, we obtain
  \[0 = \pi\left(\frac{1}{2-\gamma_A}\right) =-\frac{(\gamma_A-1)^2}{2-\gamma_A} \frac{(3 - \gamma_A)^2}{(2 - \gamma_A)^2}  \gamma_K
  + 
  \frac{(\gamma_A-1)^2}{(2 - \gamma_A)^3} (1 + \gamma_K), \]
  from which
  \begin{equation*}
    \gamma_K  = \frac{1}{(2-\gamma_A)(4-\gamma_A)} = \frac{\mu_c}{4-\gamma_A} = 
     \frac{\mu_c} {2 + \frac{1}{\mu_c}} =
     \frac{\mu_c^2} {2 \mu_c + 1}. 
  \end{equation*}
  Solving the corresponding quadratic equation for $\mu_c > 0$, we have
  \[ \mu_c = \gamma_K + \sqrt{\gamma_K^2 + \gamma_K} \le \gamma_{\max}^K + \sqrt{({\gamma_{\max}^K})^2 + \gamma_{\max}^K} ,
  \]
  and so we finally set the upper bound of $\mu_c$ for the third case to
  \[ \beta_c(\gamma_A, \gamma_K) = \min\left\{\frac{1}{2-\gamma_A}, \gamma_K + \sqrt{({\gamma_K})^2 + \gamma_K}\right\}.\]

  Summarizing, the partial derivatives of $\pi(\lambda)$ evaluated at the three roots of $\pi$ all exhibit a defined sign, except for $\frac{\partial \pi}{\partial \gamma_R} (\mu_c)$.
  The bounds derived above can be collected as in the statement of the lemma.
\end{proof}

			\begin{Corollary}
				\label{Cor:pi}
				        The eigenvalues of $\mathcal{P}^{-1} \mathcal{A}$ belong to
				\begin{equation}
					\label{bounds_0}
					\left[\mu_a(\gamma_{\min}^A,  \gamma_{\max}^R,  \gamma_{\max}^K),  \lambda_-(\gamma_{\max}^A, \gamma_{\min}^R) \right]
			\cup  I^{\pi}_+.
				\end{equation}
		\end{Corollary}
\begin{proof}
	We first show that the intervals defined for the case $\gamma_A = 1$ in \eqref{cbounds} are contained in $I^{\pi}_- \cup  I^{\pi}_+$. 
		We start from an equivalent expression for the polynomial $p(\lambda)$, which is checked by direct computation:
		\begin{equation} \label{p_alternative} p(\lambda)  = (\lambda + \gamma_R) (\lambda - \gamma_A) + \gamma_R (1 - \gamma_A) (\lambda-1).\end{equation}
		Then,
    \begin{align}
        \pi(\lambda) &= (1 + \lambda)^2 (\gamma_A - \lambda) \gamma_K + \lambda (1 + \gamma_K) p(\lambda) \nonumber \\
				&= (\lambda -\gamma_A) \left(\lambda (1 + \gamma_K)(\lambda + \gamma_R) - \gamma_K (1 + \lambda)^2   \right) +  (1 + \gamma_K) \lambda \gamma_R(1 - \gamma_A) (\lambda-1) \nonumber \\
				&= (\lambda -\gamma_A) \left(\lambda^2 + \lambda (\gamma_R(\gamma_K+1) -2 \gamma_K) - \gamma_K\right) +  (1 + \gamma_K) \lambda \gamma_R(1 - \gamma_A) (\lambda-1) \nonumber \\
				&= (\lambda -\gamma_A)  c(\lambda)  +  (1 + \gamma_K) \lambda \gamma_R(1 - \gamma_A) (\lambda-1), \label{pivsc} 
    \end{align}
                from which we can connect the roots of $c$ with the roots of $\pi$.
         Denoting as $c_-^{\min}$, $c_-^{\max}$, $c_+^{\min}$, $c_+^{\max}$ the endpoints of the intervals in \eqref{cbounds}, and recalling
		that $c_+^{\min} \le 1 \le c_+^{\max}$, we have that
  \begin{align*}
    \pi(c_-^{\min}; \gamma_{\min}^A, \gamma_{\max}^R, \gamma_{\max}^K)   &\ge 0 = \pi(\mu_a; \gamma_{\min}^A, \gamma_{\max}^R, \gamma_{\max}^K), \\
    \pi(c_-^{\max}; \gamma_{\max}^A, \gamma_{\min}^R, \gamma_{\min}^K)   &\le 0 = \pi(\mu_a; \gamma_{\max}^A, \gamma_{\min}^R, \gamma_{\min}^K),
    \intertext{showing that $\mu_a^{\min} \le c_-^{\min} < c_-^{\max} \le \mu_a^{\max}$. Furthermore,}
    \pi(c_+^{\min}; \gamma_{\min}^A, \gamma_{\max}^R, \gamma_{\min}^K)   &\le 0 = \pi(\mu_b; \gamma_{\min}^A, \gamma_{\max}^R, \gamma_{\min}^K), \\
    \pi(c_+^{\max}; \gamma_{\max}^A, \gamma_{\min}^R, \gamma_{\max}^K)   &\le 0 = \pi(\mu_c; \gamma_{\max}^A, \gamma_{\min}^R, \gamma_{\max}^K),
  \end{align*}
  showing that $\mu_b^{\min} \le c_+^{\min} \le c_+^{\max} \le \mu_c^{\max}$.

	The statement then follows from \Cref{Theo:pi} and \Cref{Lem:optpi}.
\end{proof}

	\begin{figure}[h!]
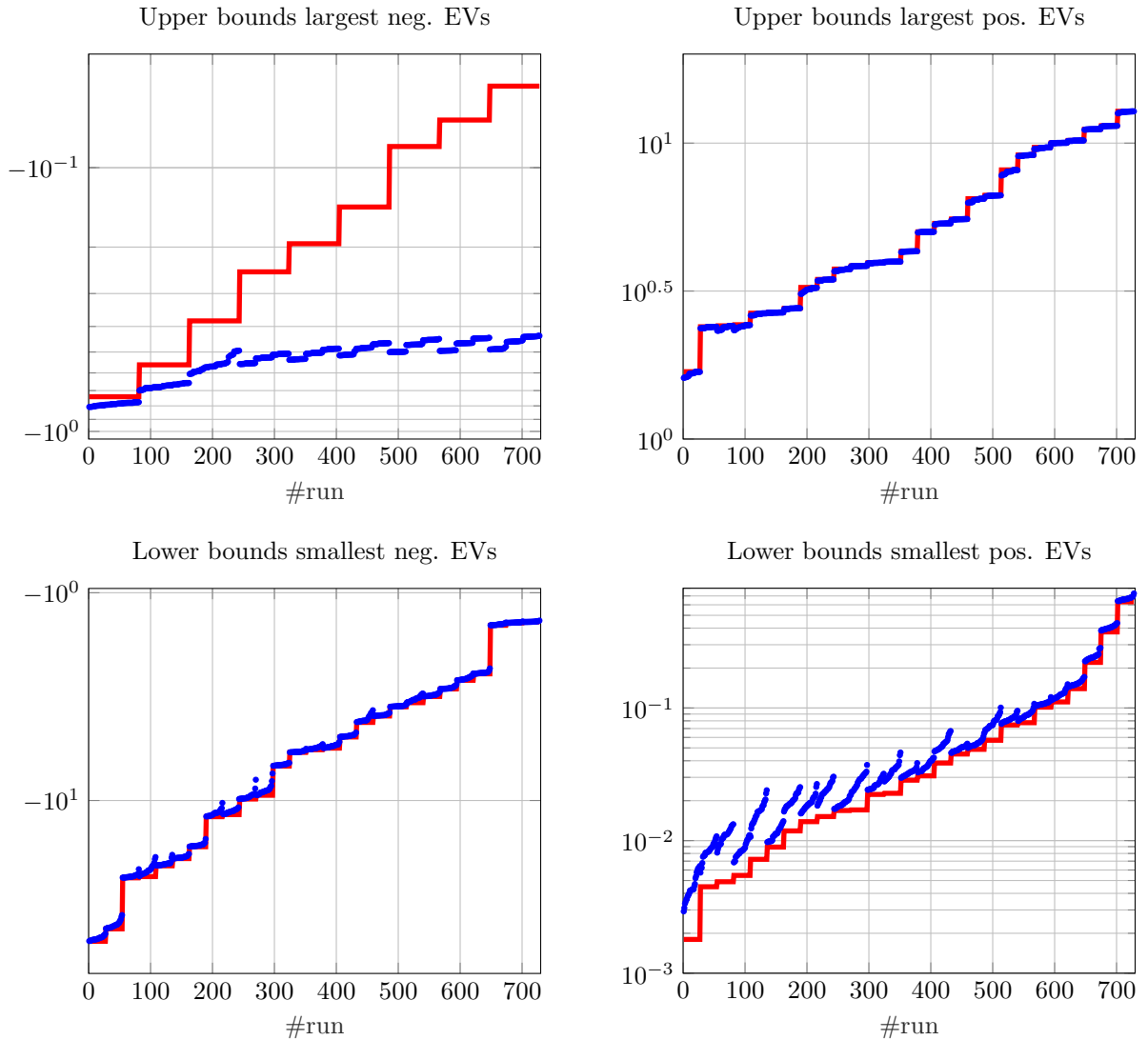

    \hspace*{-12mm}

%
%
		\caption{Extremal eigenvalues of the preconditioned matrix (blue dots) and bounds obtained from \eqref{bounds_0} (red line) after 25 runs with each combination of the parameters from \Cref{TabPar}.}
		\label{eigvsbounds}
	\end{figure}

	In \Cref{eigvsbounds} we depict the extremal eigenvalues of $\mathcal{P}^{-1} \mathcal{A}$, as compared to the developed bounds.  Further, we run $3^6 = 729$ different
	synthetic test cases combining the  values of the extremal eigenvalues of the SPD matrices involved in the previous discussion,
	reported in \Cref{TabPar}. Each test case has been run 25 times, generating random matrices which satisfy the relevant spectral properties. {In more detail, the dimensions $n$, $m$, and $p$ are computed using \texttt{60+10*rand}, using {\scshape Matlab}'s \texttt{rand} function, re-computing as necessary to ensure that $n \geq m \geq p$. The matrices $A$, $B$, and $C$ are computed using {\scshape Matlab}'s \texttt{randn} function, whereupon we take the symmetric part of $A$ and then add 1.01 times an identity matrix multiplied by the absolute value of the smallest eigenvalue, to ensure symmetric positive definiteness. We then choose $\widehat{A}$ as a linear combination of $A$ and the identity matrix, such that the eigenvalues of $\widehat{A}^{-1}A$ are contained in $[\gamma_{\min}^A,\gamma_{\max}^A]$, and similarly to construct $\widehat{S}$ and $\widehat{X}$.} In \Cref{eigvsbounds} we sort the extremal eigenvalues (and the computed bounds accordingly) for improved readability.
{We notice 
that the plots indicate (for these problems) that three bounds out of four capture
the behaviour of the eigenvalues very well, while only the upper bounds
on the negative eigenvalues are not as tight. These will be improved in Section 4 with an additional hypothesis on the sizes
of the matrices involved.}

	\begin{table}[h!]
	\begin{center}
	\begin{tabular}{|c||rrl|}
		\hline
		$	\gamma_{\min}^A$ & 0.1 & 0.3 & 0.9 \\
		$	\gamma_{\max}^A$ & 1.2 & 1.5 & 1.99\\
		$	\gamma_{\min}^R$ & 0.1 & 0.3 & 0.9 \\
		$	\gamma_{\max}^R$ & 1.2 & 1.8 & 5 \\
		$	\gamma_{\min}^K$ & 0.1 & 0.3 & 0.9 \\
		$	\gamma_{\max}^K$ & 1.2 & 1.8 & 5 \\
		\hline
		\end{tabular}
	\end{center}
		\caption{Extremal eigenvalues of $A_{\text{prec}}$, $S_{\text{prec}}$, and $X_{\text{prec}}$ used in the verification of the bounds.}
		\label{TabPar}
	\end{table}
	
\section{Eigenvalue bounds with $E \ne 0$} \label{sec:E_not_zero}
We now handle the case in which the $(3,3)$ block $E$ is nonzero.
In this case, \eqref{noE} becomes
\[	\left((1+ \lambda)^2 K Y(\lambda)^{-1} K^\top - \bar E + \lambda(I + K K^\top) \right) z =0.  \]
Proceeding then as in the proof of \Cref{Theo:pi}, we obtain that the eigenvalues of the preconditioned
matrix are the roots of the cubic polynomial
\begin{align}
    \pi_E(\lambda; \gamma_A, \gamma_R, \gamma_K, \gamma_E) &= (1+ \lambda)^2 (\gamma_A - \lambda) \gamma_K + p(\lambda) \lambda (1+\gamma_K)  - \gamma_E p(\lambda) \nonumber \\
    &= (1+ \lambda)^2 (\gamma_A - \lambda) \gamma_K + (\lambda (1+\gamma_K) -\gamma_E) p(\lambda) \nonumber \\
    &= \pi(\lambda) - \gamma_E p(\lambda). \label{piEvspi}
\end{align}
As in the case $E \equiv 0$, we analyze separately the case $\gamma_A = 1$, in which   $\lambda=1$ is a root of $\pi_E(\lambda)$.
  To this end, we write
  \begin{align}
      \pi_E(\lambda; 1, \gamma_R, \gamma_K, \gamma_E) &= \pi(\lambda; 1, \gamma_R, \gamma_K, \gamma_E) -
      \gamma_E p(\lambda; 1, \gamma_R) \nonumber \\
      &=        (\lambda - 1) \left(\lambda^2 + \lambda (\gamma_R(\gamma_K+1) -2 \gamma_K) - \gamma_K\right)
      -\gamma_E \left(\lambda^2 - (1 -\gamma_R) \lambda - \gamma_R\right) \nonumber \\
      &=        (\lambda - 1) \left(\lambda^2 + \lambda (\gamma_R(\gamma_K+1) -2 \gamma_K) - \gamma_K\right)
      -\gamma_E (\lambda-1) (\lambda + \gamma_R) \nonumber \\
      &=        (\lambda - 1) \underbrace{\left(\lambda^2 + \lambda (\gamma_R(\gamma_K+1) -2 \gamma_K - \gamma_E\right) - \gamma_K -\gamma_E \gamma_R)}_{\text{\normalsize{$c^E(\lambda)$}}}. \label{cElambda}
  \end{align}
The other two roots $c_-^E$ and $c_+^E$ of $\pi_E(\lambda; 1, \gamma_R, \gamma_K, \gamma_E)$ hence solve
$c^E(\lambda) = 0$. 

	By a similar argument as the one used for $c(\lambda)$  we can show that the smallest and largest values of the positive root of
	$c^E(\lambda)$ are separated by 1.
  \begin{figure}[bt]
    \begin{tikzpicture}
      \begin{axis}[xmin=-0.7, xmax=2.3, ymin=-1.5, ymax=2, grid=major,
          legend style={at={(0.5,0.97)}, anchor=north}, xlabel={$\lambda$}]
        \newcommand\GAMMAA{1.639}
        \newcommand\GAMMAR{0.734}
        \newcommand\GAMMAK{0.251}
        \newcommand\GAMMAE{0.412}
        \newcommand\CA{2*\GAMMAK + \GAMMAA + \GAMMAR*(1+\GAMMAK)*(\GAMMAA-2)}
        \newcommand\CB{\GAMMAR*(1+\GAMMAK) + \GAMMAK*(1-2*\GAMMAA)}
        \newcommand\CC{\GAMMAA*\GAMMAK}
        \addplot[domain=-0.7:2.3, samples=100, color2, very thick]
        {((x - (\CA)) * x - (\CB)) * x + (\CC)};
        \addlegendentry{$\pi(\lambda)$}
        \addplot[domain=-0.7:2.3, samples=100, color1, very thick]
        {-(\GAMMAE) * (x^2 - ((\GAMMAA)*(1+\GAMMAR)-2*(\GAMMAR))*x - \GAMMAR)};
        \addlegendentry{$-\gamma_E p(\lambda)$}
        \addplot[domain=-0.7:2.3, samples=100, color3, very thick]
        {((x - (\CA)) * x - (\CB)) * x + (\CC) - (\GAMMAE) * (x^2 - ((\GAMMAA)*(1+\GAMMAR)-2*(\GAMMAR))*x - \GAMMAR)};
        \addlegendentry{$\pi_E(\lambda)$}
      \end{axis}
    \end{tikzpicture}
      \caption{Polynomials $\pi(\lambda)$, $-\gamma_E\, p(\lambda)$, and $\pi_E(\lambda)$ with the same values of $\gamma_A, \gamma_R$, and $\gamma_K$ as in Figure 2,
    and $\gamma_E = 0.512$.}
    \label{FigWithE}
  \end{figure}
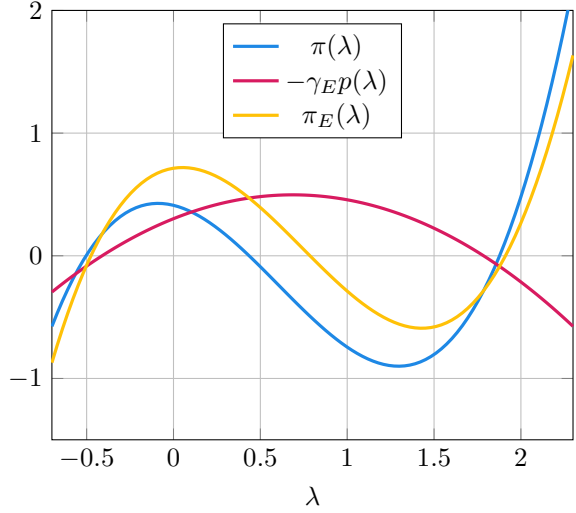
    
	Denote by
  \[
    \mu_a^E(\gamma_A, \gamma_R, \gamma_K, \gamma_E) < 0 < \mu_b^E(\gamma_A, \gamma_R, \gamma_K, \gamma_E) < \mu_c^E(\gamma_A, \gamma_R, \gamma_K, \gamma_E)
  \]
  the roots of $\pi_E(\lambda)$. 
	We are in the following situation (compare \eqref{piconditions} and \Cref{newplot}):
	\begin{equation}
		\label{Esketch}
		\begin{array}{lclclc}
			\pi_E(\mu_a) = - \gamma_E\, p(\mu_a)  \le 0, && \pi'(\mu_a) > 0 & \quad\Rightarrow\quad && \mu_a \le \mu_a^E, \\[.2em]
			\pi_E(\lambda_-) = \pi(\lambda_-) \ge  0, &&  & \quad\Rightarrow\quad && \lambda_- \ge \mu_a^E, \\[.2em]
			\pi_E(\mu_b) = - \gamma_E\, p(\mu_b)  \ge 0, && \pi'(\mu_b) < 0 & \quad\Rightarrow\quad && \mu_b \le \mu_b^E, \\[.2em]
						\pi_E(\mu_c) = - \gamma_E \,p(\mu_c)  \le 0, && \pi'(\mu_c) > 0 & \quad\Rightarrow\quad && \mu_c \le \mu_c^E ,
	\end{array}
	\end{equation}
	which shows (see also \Cref{FigWithE}) that $\mu_a^E \in [\mu_a, \lambda_-]$ and $\mu_b^E \ge \mu_b$. 
	Furthermore, it also holds that $\sgn(\pi_E'(\mu_*^E))  = \sgn(\pi'(\mu_*)) $.
It only remains to consider the upper bound for $\mu_c$. However, experimental results show that the lower bound
for $\mu_b$  may be a loose lower bound for $\mu_b^E$, so it is also of value to refine this bound. The next theorem finds two tight bounds
for the two previous quantities.
\begin{Theorem}
	\label{theo:E0}
	If $E \ne 0$ and $\gamma_A \neq 1$, the eigenvalues of $\mathcal{P}^{-1} \mathcal{A}$
        	belong to
		{
			\begin {equation}
			\label{intervalsE} \left[\mu_a(\gamma_{\min}^A,  \gamma_{\max}^R,  \gamma_{\max}^K ),  
				\lambda_-(\gamma_{\max}^A, \gamma_{\min}^R) \right]
                        \cup  I^{\pi_E}_+,
\end{equation}
where $I^{\pi_E}_+  = [\mu_l^+, \mu_u^+]$ and
\begin{align*}
	\mu_l^+ &= \min\{\gamma_{\min}^X, \mu_b^E(\gamma_{\min}^A,  \gamma_{\max}^R,  \gamma_{\min}^X, 0)  \},  \\
		 \mu_u^+ & = \max \left \{\mu_c^E(\gamma_{\max}^A,  \gamma_{\min}^R,  \gamma_{\max}^K, \gamma_{\max}^E),
		 \mu_c^E(\gamma_{\max}^A,  \gamma_{\max}^R,  \gamma_{\max}^K, \gamma_{\max}^E), \beta_c^E(\gamma^E_{\max}, \gamma^K_{\max})\right\},
\end{align*}
}
with
   \begin{equation}
     \label{betac}
     \beta_c^E(\gamma_E, \gamma_K) = 
     \gamma_K +\frac{\gamma_E}{2} + 
     \sqrt{\left(\gamma_K + \frac{\gamma_E}{2}\right)^2 + \gamma_K}.
   \end{equation}
\end{Theorem}
\begin{proof}
	\noindent \underline{Step 1:} (Bounding $\mu_c^E$ from above.) As in Figure \ref{FigWithE}, we have that $\frac{\partial \pi_E}{\partial \lambda}(\mu_c^E) > 0$. {We use \eqref{newplambda} to rearrange}
  \begin{align*}
    \pi_E(\lambda) &= \pi(\lambda) - \gamma_E p(\lambda) 
    = (1+\lambda)^2 (\gamma_A - \lambda) \gamma_K + p(\lambda) (\lambda (1 + \gamma_K) - \gamma_E) \\
    &= (\gamma_A - \lambda)\left[ (1+\lambda)^2 \gamma_K - \lambda (\gamma_R + 1) (\lambda (\gamma_K + 1) - \gamma_E) \right] - \gamma_R (\lambda - 1)^2 (\lambda (1 + \gamma_K) - \gamma_E),
  \end{align*}
  which, by affine linearity in $\gamma_A$, implies
  \begin{align*}
	  \frac{\partial \pi_E}{\partial \gamma_A}(\lambda) &= \frac{{\pi_E}(\lambda) + \gamma_R (\lambda - 1)^2 \left((1+\gamma_K) \lambda -\gamma_E\right)}{\gamma_A - \lambda}.
    \intertext{The remaining partial derivatives of $\pi_E$ are easily obtained from those of $\pi$ as}
    \frac{\partial \pi_E}{\partial \gamma_R}(\lambda) &= \left( (1 + \gamma_K)\lambda - \gamma_E \right) \left( (2-\gamma_A) \lambda - 1 \right), \\
		\frac{\partial \pi_E}{\partial \gamma_K}(\lambda) &= \frac{\partial \pi}{\partial \gamma_K}(\lambda) = \frac{\pi(\lambda) - \lambda p(\lambda)}{\gamma_K} =
		\frac{\pi_E(\lambda) + (\gamma_E-\lambda) p(\lambda)}{\gamma_K},\\
    \frac{\partial \pi_E}{\partial \gamma_E}(\lambda) &= -p(\lambda).
  \end{align*}
  We now show that $(\gamma_K + 1) \mu^E_c - \gamma_E >0$. In fact 
  from the sketch \eqref{Esketch} we have that $\mu_c^E \ge  \mu_c > \gamma_A$, 
  meaning that if 
  $\gamma_A \ge \frac{\gamma_E}{1+\gamma_K}$ then it also holds that 
  $\mu_c^E > \frac{\gamma_E}{1+\gamma_K}$. If instead
  $\gamma_A < \frac{\gamma_E}{1+\gamma_K}$, then since
  \[\pi_E\left(\frac{\gamma_E}{1+\gamma_K}\right) = 
  \left(1 + \frac{\gamma_E}{1+\gamma_K}\right)^2 \left(\gamma_A - \frac{\gamma_E}{1+\gamma_K} \right) \gamma_K < 0,\]
  it must again hold that $\frac{\gamma_E}{1+\gamma_K} < \mu_c^E$.

  Hence, we obtain that
  \begin{align*}
    \frac{\partial \pi_E}{\partial \gamma_A}(\mu^E_c) &< 0, &
    \frac{\partial \pi_E}{\partial \gamma_E}(\mu^E_c) &< 0.
  \end{align*}
  Lemma~\ref{lem:extremal_roots} then delivers that $\gamma_A = \gamma^A_{\max}$ and $\gamma_E = \gamma^E_{\max}$ if $\mu_c$ is a local maximum.
  
  The partial derivative $\frac{\partial \pi_E}{\partial \gamma_R}(\mu_c^E)$ 
  has the same sign as
  $(2-\gamma_A) \mu_c^E - 1$.
  Let us consider the case $ \mu_c^E = \frac{1}{2 - \gamma_A}$, and write
  \[ 0 = \pi_E\left(\frac{1}{2 - \gamma_A}\right)  = \frac{(\gamma_A-1)^2}{(2-\gamma_A)^2}\left(-\frac{(3 - \gamma_A)^2}{2-\gamma_A} \gamma_K + \frac{1 + \gamma_K}{2 - \gamma_A} - \gamma_E\right).\]
  Observing that $3 - \gamma_A = 1 + \frac {1} {\mu_c^E}$, we rewrite the previous identity as
  \[ -\gamma_K \left(1 + \frac {1}{\mu_c^E} \right)^2 \mu_c^E + (1 + \gamma_K) \mu_c^E - \gamma_E =0, \]
  and finally as
  \[ (\mu_c^E)^2 - (2\gamma_K + \gamma_E) \mu_c^E - \gamma_K = 0.\]
  Therefore, in the case $\frac{\partial \pi_E}{\partial \gamma_R}(\mu_c) = 0$, we can bound $\mu_c^E \le \beta^E_c(\gamma^E_{\max}, \gamma^K_{\max})$. In the other two cases, we can use the worse case of $\gamma_R \in \{ \gamma^R_{\min}, \gamma^R_{\max} \}$.

  We now turn to
  \[ 
    \frac{\partial \pi_E}{\partial\gamma_K}(\mu_c^E) = \frac{(\gamma_E-\mu_c^E) p(\mu_c^E)}{\gamma_K},
  \]
  which is negative if $\mu_c^E > \gamma_E$ (implying that $\gamma_K = \gamma^K_{\max}$ if $\mu^E_c$ is a local maximum, by Lemma~\ref{lem:extremal_roots}). Otherwise, we have the bound
  $\mu_c^E  \le \gamma_{\max}^E$, which is dominated by $\beta^E_c(\gamma^E_{\max}, \gamma^K_{\max})$; see \eqref{betac}.

  Summarizing, an upper bound for the largest positive root of $\pi(\lambda)$ is given by
  \[\max\left\{
  \mu_c^E(\gamma_{\max}^A,  \gamma_{\min}^R,  \gamma_{\max}^K, \gamma_{\max}^E),
  \mu_c^E(\gamma_{\max}^A,  \gamma_{\max}^R,  \gamma_{\max}^K, \gamma_{\max}^E), \beta_c^E(\gamma^E_{\max}, \gamma^K_{\max})\right\}. 
  \]

  \noindent \underline{Step 2:} (Bounding $\mu^E_b$ from below.)
  To obtain a tight bound, we recall that
  $\gamma_K = \gamma_X - \gamma_E$ and define
  \begin{align}
	  \label{w_vs_pi}
    w(\lambda; \gamma_A, \gamma_R, \gamma_X, \gamma_E)  
    &\equiv (1+ \lambda)^2 (\gamma_A - \lambda) \gamma_X + \lambda p(\lambda) (1+\gamma_X) - \gamma_E \left( (1+\lambda)^2 (\gamma_A - \lambda) + p(\lambda) (1 + \lambda)\right) \nonumber \\
	  &= (1+ \lambda)^2 (\gamma_A - \lambda) (\gamma_X - \gamma_E) + \lambda p(\lambda) (1+\gamma_X - \gamma_E) - \gamma_E p(\lambda) =\nonumber 
	  \\ &= \pi_E(\lambda, \gamma_A, \gamma_R, \gamma_X- \gamma_E, \gamma_E) .
  \end{align}
  Aiming towards the application of Lemma~\ref{lem:extremal_roots} to $w$, we immediately observe that $\frac{\partial w}{\partial \lambda}(\mu^E_b) < 0$ and that
  \begin{align}
		\nonumber \frac{\partial w}{\partial \gamma_A}(\lambda) 
    &= \frac{\partial \pi_E}{\partial \gamma_A}(\lambda)
	  = \frac{{\pi_E}(\lambda) + \gamma_R (\lambda - 1)^2 \left(\lambda (1+\gamma_K) - \gamma_E\right)}{\gamma_A - \lambda}, \\
		\nonumber\frac{\partial w}{\partial \gamma_R}(\lambda) 
		\label{nablaR_w}
    &= \frac{\partial \pi_E}{\partial \gamma_R}(\lambda) 
	  = \left((1 + \gamma_K)\lambda - \gamma_E\right) (\lambda (2-\gamma_A) - 1), \\[-.8em] 
	  &&\\[-.8em]
		\nonumber \frac{\partial w}{\partial \gamma_X}(\lambda) 
    &= \frac{\partial \pi_E}{\partial \gamma_K}(\lambda)
	  = \frac{\pi_E(\lambda) + (\gamma_E-\lambda) p(\lambda)}{\gamma_K}, \\
		  \nonumber 
		  \frac{{\rm d} w}{{\rm  d} \gamma_E}(\lambda)  &\overset{\eqref{w_vs_pi}}{=}
	  -\frac{\partial \pi_E}{\partial \gamma_K}(\lambda) + \frac{\partial \pi_E}{\partial \gamma_E}(\lambda) \\
	  \nonumber &= -\frac{\pi_E(\lambda) + (\gamma_E -\lambda)}{\gamma_K} \, p(\lambda) - p(\lambda) =
	  \frac{-\pi_E(\lambda) + (\lambda - \gamma_X)}{\gamma_K} \, p(\lambda).
  \end{align}
  We now consider the expression $\mu_b^E (1+\gamma_K) - \gamma_E $.  We recall that $\mu_b^E$ must be smaller than $\gamma_{\min}^A$,
  since the polynomial $\pi_E(\lambda) = w(\lambda)$ is obtained by assuming $\lambda \not \in I_-\cup I_+$.
  The condition $\mu_b^E < \gamma_A$ is equivalent to $w(\gamma_A) < 0$, since $\pi_E$ is decreasing in $\lambda$ around $\mu_b^E$.
  Now observing that
  \[ w(\gamma_A) = -\gamma_R(\gamma_A-1)^2 (\gamma_A (1+\gamma_K) - \gamma_E ),\]
  we obtain that $\mu_b^E < \gamma_A$ implies 
  \[
    \gamma_A > \frac {\gamma_E}{1+\gamma_K}.
  \]
  This condition also implies that $\mu_b^E > \frac{ \gamma_E}{1+\gamma_K}$, due to 
  \[\pi_E\left(\frac{\gamma_E}{1+\gamma_K}\right) = 
  \left(1 + \frac{\gamma_E}{1+\gamma_K}\right)^2 \left(\gamma_A - \frac{\gamma_E}{1+\gamma_K} \right) \gamma_K > 0.\]
  We have proved that $\frac{\gamma_E}{1 + \gamma_K} < \mu_b^E < \gamma_A$, which provides
	\[
    \frac{\partial w}{\partial \gamma_A}(\mu_b^E) =
		\frac{\gamma_R (\mu_b^E - 1)^2 \left(\mu_b^E (1+\gamma_K) - \gamma_E\right)}{\gamma_A - \mu_b^E} > 0.
  \]
  Thus, if $\mu^E_b$ is a local minimum, then $\gamma_A = \gamma^A_{\min}$ by Lemma~\ref{lem:extremal_roots}.

  Now, the inequality 
  \[\mu_b^E(2 -\gamma_A) -1 < \gamma_A (2 -\gamma_A) -1 =-(\gamma_A-1)^2 < 0\]
  together with the previous discussion yields that
  \[
    \frac{\partial w}{\partial \gamma_R}(\mu_b^E) =
    \left[(1 + \gamma_K)\mu_b^E - \gamma_E\right] (\mu_b^E (2-\gamma_A) - 1) < 0.
  \]
  Lemma~\ref{lem:extremal_roots} yields that $\gamma_R = \gamma^R_{\max}$ if $\mu^E_b$ is a local minimum.

	{Now, the total derivative 
		\[ \frac{{\rm d} w}{{\rm d}  \gamma_E}(\mu_b^E) = \frac{\mu_b^E - \gamma_X}{\gamma_K} \, p(\mu_b^E)\]
		is positive for $\mu_b^E < \gamma_X$ (alternatively we have the bound $\mu_b^E \ge \gamma_{\min}^X$). 
		Using Lemma~\ref{lem:extremal_roots}, a lower bound for $\mu_b^E < \gamma_X$ can therefore be obtained with $\gamma_E \equiv 0$. With this value we determine the sign of the partial derivative with respect to $\gamma_X$:
		\[
      \frac{\partial w}{\partial \gamma_X}(\mu_b^E, \gamma_{\min}^A, \gamma_{\max}^R, \gamma_X, 0) = -\frac{\mu_b^E p(\mu_b^E)}{\gamma_K}  > 0.
    \]
		Hence, a lower bound for the positive root is
		\[ \min\{\gamma_{\min}^X, \mu_b^E(\gamma_{\min}^A, \gamma_{\max}^R, \gamma_{\min}^X, 0)\}.\]
		}
    Summarizing the results of Steps 1 and 2 yields the assertion.
  \end{proof}
  
  The next corollary states that the findings of \Cref{theo:E0} also hold when $\gamma_A = 1$.
   \begin{Corollary}
                                \label{Cor:pi_E}
                                        The eigenvalues of $\mathcal{P}^{-1} \mathcal{A}$, with $E \ne 0$  belong to
					         \[\left[\mu_a(\gamma_{\min}^A,  \gamma_{\max}^R,  \gamma_{\max}^K, ),
                                \lambda_-(\gamma_{\max}^A, \gamma_{\min}^R) \right]
                        \cup  I^{\pi_E}_+.
                        \]
                \end{Corollary}
\begin{proof}
	It is sufficient to prove that the intervals characterizing  the roots of $c^E(\lambda)$ are contained in those
	defined by \eqref{intervalsE}.
	We denote
		as $c_{l,-}^{E}$, $c_{u,-}^{E}$, $c_{l,+}^{E}$, $c_{u,+}^{E}$ the bounds for the roots of $c^E(\lambda)$.
		To express $\pi_E(\lambda)$ in terms of  $c^E(\lambda)$, 
	we first observe from \eqref{cElambda} that
	\[ c^E(\lambda) = c(\lambda) - \gamma_E (\lambda + \gamma_R),\]
		then we use \eqref{piEvspi} and \eqref{pivsc} to write 
	\begin{align*}
		\pi_E(\lambda) &= \pi(\lambda) - \gamma_E p(\lambda)   \\
		 &=  (\lambda - \gamma_A) c(\lambda) 
		+ (1 + \gamma_K) \lambda \gamma_R (1 - \gamma_A) (\lambda-1)
		- \gamma_E p(\lambda) \\
		&= 	(\lambda -\gamma_A) c^E(\lambda) 
		+ \gamma_E (\lambda + \gamma_R)(\lambda -\gamma_A)  - \gamma_E p(\lambda)
		+ (1 + \gamma_K) \lambda \gamma_R (1 - \gamma_A) (\lambda-1)\\
		&= 	(\lambda -\gamma_A) c^E(\lambda) 
		+ \gamma_E \left((\lambda + \gamma_R)(\lambda -\gamma_A)  - p(\lambda) \right)
		+ (1 + \gamma_K) \lambda \gamma_R (1 - \gamma_A) (\lambda-1).
	\end{align*}
	Using \eqref{p_alternative}, 
we finally obtain that
		\[ \pi_E(\lambda) =  (\lambda -\gamma_A) c^E(\lambda)
	+ [(1 + \gamma_K) \lambda - \gamma_E] \gamma_R (1 - \gamma_A) (\lambda-1).
	\]
	 The signs of $\pi_E(\lambda)$ in $c_{l,-}^{E}$, $c_{u,-}^{E}$, $c_{l,+}^{E}$, $c_{u,+}^{E}$ are obtained by observing that $c^E\left(\frac{\gamma_E}{1 + \gamma_K}\right) < 0$, which implies that $\frac{\gamma_E}{1 + \gamma_K} < c_+^E$, and by making use of the following sketch:
\begin{center}
\begin{tabular}{|r||ccc|} \hline
	$\lambda =$ & $c_{\{l,u\},-}^{E}$ & $c_{l,+}^{E}$ & $c_{u,+}^{E}$ \\ \hline \hline
  $\lambda$ & $-$ & $+$ & $+$ \\
  $\lambda-1$ & $-$ & $-$ & $+$ \\
  $(1 + \gamma_K)\lambda - \gamma_E$ & $-$ & $+$ & $+$ \\ \hline
\end{tabular}
\end{center}
\begin{align*}
  \pi_E(c_{l,-}^E; \gamma_{\min}^A, \gamma_{\max}^R, \gamma_{\max}^K, \gamma_{\min}^E)   & > 0 = \pi_E(\mu_a^E; \gamma_{\min}^A, \gamma_{\max}^R, \gamma_{\max}^K, \gamma_{\min}^E), \\
  \pi_E(c_{u,-}^E; \gamma_{\max}^A, \gamma_{\min}^R, \gamma_{\min}^K, \gamma_{\max}^E)   & < 0 = \pi_E(\mu_a^E; \gamma_{\max}^A, \gamma_{\min}^R, \gamma_{\min}^K, \gamma_{\max}^E), \\
  \pi_E(c_{l,+}^E; \gamma_{\min}^A, \gamma_{\max}^R, \gamma_{\min}^K, \gamma_{\min}^E)   & < 0 =  \pi_E(\mu_b^E; \gamma_{\min}^A, \gamma_{\max}^R, \gamma_{\min}^K, \gamma_{\min}^E), \\
  \pi_E(c_{u,+}^E; \gamma_{\max}^A, \gamma_{\min}^R, \gamma_{\max}^K, \gamma_{\max}^E)   & < 0 =  \pi_E(\mu_c^E; \gamma_{\max}^A, \gamma_{\min}^R, \gamma_{\max}^K, \gamma_{\max}^E),
\end{align*}
showing that the intervals bounding the roots of $c^E(\lambda)$ are contained within those locating
the roots of $\pi_E(\lambda)$.
\end{proof}

\section{Refined upper bound when $C$ is invertible.} \label{sec:C_invertible}

We now consider the setting where $m = p$ and $C$ is invertible, in both cases $E = 0$ and $E \ne 0$.

\begin{Theorem}
        \label{theo_refined}
        If $C$ is  square and nonsingular, then any eigenvalue $\lambda$ of
	of $\mathcal{P}^{-1} \mathcal{A}$ not lying in $I_A$ is a root of $\pi(\lambda) = 0$ ($\pi_E(\lambda) = 0$ if $E \ne 0$), for a suitable value of $\gamma_A$, $\gamma_R$, $\gamma_K$, $\gamma_X$, and $\gamma_E$.
\end{Theorem}
        \begin{proof}
		Let $E \equiv 0$. We start by obtaining an expression for $z$ from \eqref{Eq34.1},
		\begin{equation*} 
		z = \frac{1+\lambda}{\lambda}\left(I + K K^{\top}\right )^{-1}K y,
		\end{equation*}
			and substitute it into \eqref{Zlambda}, yielding
			\[ 
		Y(\lambda)y + 
			  \frac{(1+\lambda)^2}{\lambda}K^\top\left(I + K K^{\top}\right )^{-1}K y  = 0. \]
			  By pre-multiplying the previous by $\frac{y^\top}{y^\top y} $, we obtain
			\begin{equation} \label{pinew} 
		0 = \frac{y^\top Y(\lambda)y}  
		{y^\top y}+
			  \frac{\frac{(1+\lambda)^2}{\lambda}y^\top K^\top\left(I + K K^{\top}\right )^{-1}K y}{y^\top y} 
				= \frac{p(\lambda)}{\gamma_A - \lambda} +
				\frac{(1+\lambda)^2}{\lambda}\frac{y^\top \left(I + (K^{\top} K)^{-1}\right )^{-1} y}{y^\top y},
			\end{equation}
			since $K^\top K$ is invertible by hypothesis.
		Applying now Lemma \ref{Le1} to $I + (K^{\top}K)^{-1}$ we obtain, for a suitable vector $s \ne 0$, 
\[ \frac{y^\top \left(I + (K^{\top} K)^{-1}\right )^{-1} y}{y^\top y} =  \left[\frac{s^\top (I + (K^\top K)^{-1}) s}{s^\top s}\right]^{-1} \overset{(v = K^{-\top} s)}{=} 
				 \left[1 + \frac{v^\top v}{v^\top K K^\top v}\right] ^{-1} =
				\left[1 + \frac{1}{\gamma_k}\right]^{-1}  = \frac{\gamma_K}{1 + \gamma_K}.\]
We can therefore rewrite \eqref{pinew} as
\begin{equation}
	\label{pinew_1}
	\frac{p(\lambda)}{\gamma_A-\lambda} + \frac{(1+\lambda)^2}{\lambda} \frac{\gamma_K}{1 + \gamma_K} = 0. 
\end{equation}
		Rearranging the  terms in \eqref{pinew_1} yields the usual polynomial equation $\pi(\lambda) = 0.$ \\[.5em]
		\noindent
			When $E \ne 0$, the counterpart of \eqref{Eq34.1} reads 
		\begin{equation*} 
			\left(\lambda(I + K K^{\top}) - \bar E\right) z = (1+\lambda) K y.
		\end{equation*}
                We have previously shown (see beginning of Section \ref{sec:E_not_zero})
  that the positive eigenvalues of $\mathcal{P}^{-1} \mathcal{A}$ are the roots of $\pi_E(\lambda) = 0$ and that they
  are bounded by $I_+^{\pi_E}$, as stated in \Cref{Cor:pi_E}. \\
                Let us now assume
		$\lambda < 0$. The matrix on the right-hand side is negative definite, so we can write
		\[ z =   \left(\lambda(I + K K^{\top}) - \bar E\right)^{-1} (1+\lambda) K y. \]
		Upon substitution of the previous into \eqref{Zlambda} and pre-multiplication by $\frac{y^\top}{y^\top y} $ we obtain
		 \begin{align} \label{piEnew}
                0 ={}& \frac{y^\top Y(\lambda)y}
                {y^\top y}+
			  \frac{(1+\lambda)^2 y^\top K^\top   \left(\lambda(I + K K^{\top}) - \bar E\right)^{-1}K y}{y^\top y} \nonumber \\
			    ={}&\frac{p(\lambda)}{\gamma_A-\lambda} + 
				(1+\lambda)^2\frac{y^\top (\overbrace{\lambda(I + (K^{\top} K)^{-1}) - K^{-1} \bar E
				K^{-\top}}^{Q(\lambda)} )^{-1} y}{y^\top y}  =
			 \frac{p(\lambda)}{\gamma_A-\lambda} + 
				 (1+\lambda)^2\frac{y^\top Q(\lambda)^{-1} y}{y^\top y}.
		 \end{align}
We now apply Lemma \ref{Le1} to the matrix function $Q(\lambda)$, obtaining, for a suitable nonzero vector $s$,
\[ \frac{y^\top Q(\lambda)^{-1} y}{y^\top y} =  \left[\lambda \, \frac{s^\top (I + (K^\top K)^{-1}) s}{s^\top s}
- \frac{s^\top K^{-1} \bar E K^{-\top}s}{s^\top s}\right]^{-1}  = 
				\left[\lambda\left(1 + \frac{1}{\gamma_K}\right) - \frac{\gamma_E}{\gamma_K} \right]^{-1}  = \frac{\gamma_K}{\lambda(1 + \gamma_K) - \gamma_E}.\]
We can therefore rewrite \eqref{piEnew} as
\begin{equation*}
	\frac{p(\lambda)}{\gamma_A-\lambda} + (1+\lambda)^2 \frac{\gamma_K}{\lambda(1 + \gamma_K) - \gamma_E} = 0. 
\end{equation*}
		Rearranging the  terms yields the usual polynomial equation $\pi_E(\lambda) = 0.$
\end{proof}

		{
The previous result allows us to conclude that the eigenvalues of $\mathcal{P}^{-1} \mathcal{A}$ are the roots
of $\pi(\lambda) = 0$ ($\pi_E(\lambda) = 0$), not lying in
	$[\gamma_{\min}^A,\gamma_{\max}^A]$. Hence, the upper bound for $\mu_a$ provided by \Cref{Lem:optpi} 
	is also an upper bound for the negative eigenvalues of the preconditioned matrix, as stated below.
	}

	\begin{Corollary}
		\label{Cor:m=p}
		Let $E \equiv 0$. Then the eigenvalues of $\mathcal{P}^{-1} \mathcal{A}$ not lying in $I_A$  are contained in 
                       $ I^{\pi}_- \cup 
                        I^{\pi}_+  $.
	\end{Corollary}

	\begin{Corollary}
		\label{Cor:Csquare}
		Let $E  \ne 0$. Then the eigenvalues of $\mathcal{P}^{-1} \mathcal{A}$ not lying in $I_A$  are contained in 
		       \[ \left[ \mu_a^E(\gamma_{\min}^A, \gamma_{\max}^R, \gamma_{\max}^X, \gamma_{\min}^E),
\mu_a^E(\gamma_{\max}^A, \gamma_{\min}^R, \gamma_{\min}^X, \gamma_{\min}^X - \gamma_{\min}^K )\right]
		       \cup 
			 I^{\pi_E}_+ . \]
	\end{Corollary}

\begin{proof}
	     {
        To improve the  upper bound of the negative eigenvalues, we consider
        the partial derivatives of $w(\lambda)$,  as displayed in \eqref{nablaR_w}.}
  From \Cref{sign}, we deduce that $\frac{\partial \pi_E}{\partial \lambda}(\mu^E_a) > 0$, and the signs of the other partial derivatives evaluated at $\mu^E_a < 0$ satisfy
  \begin{align*}
    \frac{\partial w}{\partial \gamma_A}(\mu_a^E) &=
    \frac{\gamma_R (\mu_a^E - 1)^2 \left(\mu_a^E (1+\gamma_K) - \gamma_E\right)}{\gamma_A - \mu_a^E} < 0, \\
    \frac{\partial w}{\partial \gamma_R}(\mu_a^E) &= \left[(1 + \gamma_K)\mu_a^E - \gamma_E\right] (\mu_a^E (2-\gamma_A) - 1) > 0, \\
        \frac{\partial w}{\partial \gamma_X}(\mu_a^E) &= \frac{(\gamma_E-\mu_a^E) p(\mu_a^E)}{\gamma_K} > 0, \\
	  \frac{{\rm d} w}{{\rm d} \gamma_E}(\mu_a^E)   &= 
		    \frac{\mu_a^E - \gamma_X}{\gamma_K}p(\mu_a^E) < 0.
  \end{align*}
	Applying once again Lemma \ref{lem:extremal_roots}, the
	maximum of $\mu_a^E$ is obtained for $\gamma_E$ equal
        to its maximum value, which, in this case, is not necessarily $\gamma_{\max}^E$. In fact,
        after the change of variables $\gamma_K = \gamma_X - \gamma_E$, the indicator
        $\gamma_E$ must satisfy
        $0 < \gamma_{\min}^K \le \gamma_K = \gamma_X - \gamma_E$.
        Since
        \[w(\lambda; \gamma_{\max}^A, \gamma_{\min}^R, \gamma_{\min}^X, \gamma_E) \equiv
        \pi_E(\lambda; \gamma_{\max}^A, \gamma_{\min}^R, \gamma_{\min}^X-\gamma_E, \gamma_E),  \]
                       the maximum value of $\gamma_E$ does not exceed  $\min\{\gamma_{\max}^E,
                       \gamma_{\min}^X - \gamma_{\min}^K \} = \gamma_{\min}^X - \gamma_{\min}^K$, as in the assertion of the corollary.
%
\end{proof}

	\begin{figure}[h!]
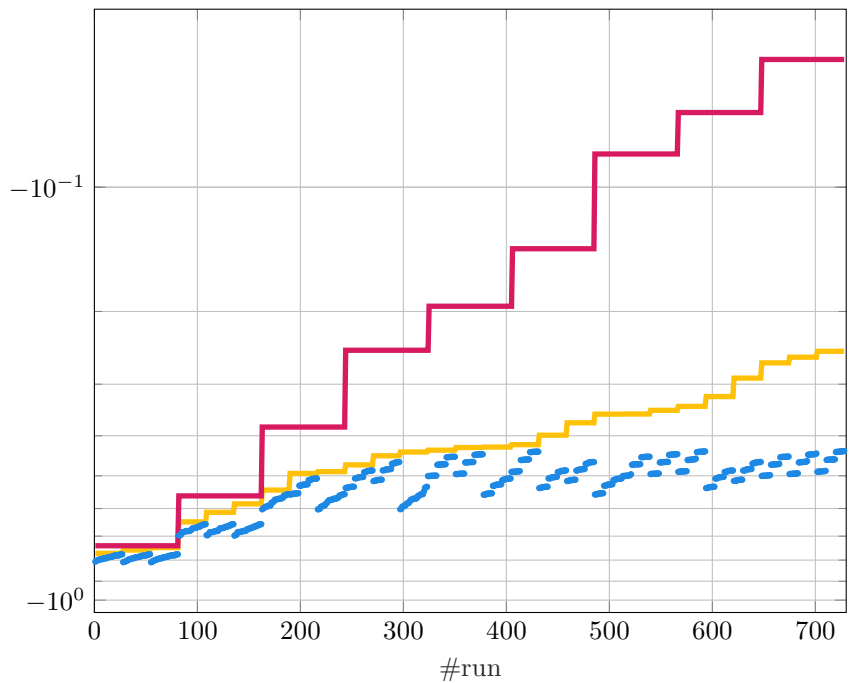

	\begin{center}

%
%
%
		\caption{Comparisons between the bounds based on $\lambda_-$ (red line),  and the refined upper bounds from 
		\Cref{Cor:m=p} (yellow line) 
		for the negative eigenvalues. Case with $E \equiv 0$ and a square invertible matrix $C$ (compare with \Cref{eigvsbounds},
		top-left plot).}
		\label{newbounds}
	\end{center}
	\end{figure}
	We conclude this section by showing a graphical interpretation of the bounds just developed. In Figure \ref{newbounds}
	we report the results  obtained by running 25 times each test case of Table \ref{TabPar}, and imposing
	that $m = p$. The negative eigenvalues are reported, as well as  the bounds provided by Corollary \ref{Cor:pi}
	and the bounds stated in Corollary \ref{Cor:m=p}.

	\section{Numerical Experiments} \label{sec:numerics}

We now seek to validate numerically our eigenvalue bounds for preconditioned double saddle-point systems, through two model problems from PDE-constrained optimization. For a range of problem setups, we present extremal (negative and positive) eigenvalues of $\mathcal{P}^{-1} \mathcal{A}$, along with theoretical bounds. All matrices are generated in Python using the package \cite{skfem2020}, and we solve the resulting systems in {\scshape Matlab} R2018a. For reference, we also provide the iteration numbers required for the solution of the relevant systems with {\scshape Matlab}'s inbuilt preconditioned {\scshape Minres} \cite{minres} routine to tolerance $10^{-10}$. All tests are carried out on an Intel(R) Core(TM) i7-6700T CPU @ 2.80GHz quad-core processor.

The problems we consider are of the form
\begin{align*}
\min_{y,u} \quad &\mathcal{J}(y,u) \\
\text{s.t.} \quad &\left\{\begin{array}{rl}
-\Delta y+y+u=0 & \text{in }\Omega, \\
\frac{\partial y}{\partial n}=0 & \text{on }\partial\Omega, \\
\end{array}\right.
\end{align*}
where $\Omega$ is a domain with boundary $\partial\Omega$. Here, $y$ and $u$ denote \emph{state} and \emph{control variables}. The cost functional $\mathcal{J}(y,u)$ may be of the form
\begin{equation*}
\mathcal{J}_{\Omega}(y,u) = \frac{1}{2}\left\|y-\widehat{y}\right\|_{L^2(\Omega)}^2+\frac{\beta}{2}\left\|u\right\|_{L^2(\Omega)}^2 \quad \text{or} \quad \mathcal{J}_{\partial\Omega}(y,u) = \frac{1}{2}\left\|y-\widehat{y}\right\|_{L^2(\partial\Omega)}^2+\frac{\beta}{2}\left\|u\right\|_{L^2(\Omega)}^2,
\end{equation*}
with $\widehat{y}$ a specified \emph{desired state}, and $\beta>0$ a regularization parameter. With $\mathcal{J} = \mathcal{J}_{\Omega}$ we refer to this as a \emph{full observation problem}, and with $\mathcal{J} = \mathcal{J}_{\partial\Omega}$ we denote this as a \emph{boundary observation problem}. In the subsequent results, we consider $\Omega = (0,1)^2$ and discretize these problems using P1 finite elements, although there is considerable flexibility as to the setup we could select. We highlight that it would also be perfectly possible to solve either the full or boundary observation problems by tackling a classical (generalized) saddle-point system; however, as the objective of this work is to examine the behaviour of the eigenvalues of preconditioned double saddle-point systems, we follow this approach. We note that both full and boundary observation problems lead to systems where $E \neq 0$ and $C$ is invertible, and we will use this for our interpretation of our analytic bounds.

\subsection{Full Observation Problem}

\begin{table}[b]
\centering
\caption{Computed eigenvalues of $\mathcal{P}^{-1} \mathcal{A}$ and bounds, for full observation problem with $h = 2^{-4}$, $\beta = 10^{-2}$, and a range of Chebyshev semi-iterations $\ell$. Results are presented in the following order, from left to right: lower bound on negative eigenvalues, computed smallest negative eigenvalue, computed largest negative eigenvalue, upper bound on negative eigenvalues, lower bound on positive eigenvalues, computed smallest positive eigenvalue, computed largest positive eigenvalue, upper bound on positive eigenvalues.}\label{Table1}
\begin{tabular}{|c||c|c||c|c||c|c||c|c|}
\hline
$\ell$ & $\text{Bound}_l^-$ & $\rho_l^-$ & $\rho_u^-$ & $\text{Bound}_u^-$ & $\text{Bound}_l^+$ & $\rho_l^+$ & $\rho_u^+$ & $\text{Bound}_u^+$ \\ \hline \hline
1 & $-8.4013$ & $-1.3980$ & $-0.4969$ & $-0.0979$ & 0.0331 & 0.3518 & 2.7187 & 5.1949 \\ \hline
2 & $-2.2801$ & $-1.2133$ & $-0.7622$ & $-0.4926$ & 0.2355 & 0.6118 & 1.6288 & 3.9195 \\ \hline
3 & $-1.2769$ & $-1.0740$ & $-0.9266$ & $-0.7966$ & 0.4718 & 0.6570 & 1.3591 & 3.0825 \\ \hline
4 & $-1.0796$ & $-1.0247$ & $-0.9750$ & $-0.9280$ & 0.5867 & 0.6594 & 1.3321 & 2.8085 \\ \hline
5 & $-1.0253$ & $-1.0082$ & $-0.9918$ & $-0.9755$ & 0.6283 & 0.6596 & 1.3299 & 2.7199 \\ \hline
7 & $-1.0083$ & $-1.0009$ & $-0.9991$ & $-0.9918$ & 0.6424 & 0.6596 & 1.3297 & 2.6883 \\ \hline
10 & $-1.0028$ & $-1.0000$ & $-1.0000$ & $-0.9973$ & 0.6472 & 0.6596 & 1.3297 & 2.6804 \\ \hline
\end{tabular}
\end{table}

\begin{table}[b]
\centering
\caption{Computed eigenvalues of $\mathcal{P}^{-1} \mathcal{A}$ and bounds, for full observation problem with $h = 2^{-4}$, $\beta = 10^{-4}$, and a range of Chebyshev semi-iterations $\ell$.}\label{Table2}
\begin{tabular}{|c||c|c||c|c||c|c||c|c|}
\hline
$\ell$ & $\text{Bound}_l^-$ & $\rho_l^-$ & $\rho_u^-$ & $\text{Bound}_u^-$ & $\text{Bound}_l^+$ & $\rho_l^+$ & $\rho_u^+$ & $\text{Bound}_u^+$ \\ \hline \hline
1 & $-7.3358$ & $-1.3634$ & $-0.4953$ & $-0.0979$ & 0.0334 & 0.3506 & 2.0723 & 4.2244 \\ \hline
2 & $-2.2530$ & $-1.2119$ & $-0.7619$ & $-0.4926$ & 0.2372 & 0.6455 & 1.5477 & 3.7137 \\ \hline
3 & $-1.2761$ & $-1.0739$ & $-0.9261$ & $-0.7966$ & 0.4757 & 0.6609 & 1.3113 & 3.0035 \\ \hline
4 & $-1.0796$ & $-1.0247$ & $-0.9750$ & $-0.9280$ & 0.5916 & 0.6612 & 1.3029 & 2.7578 \\ \hline
5 & $-1.0253$ & $-1.0082$ & $-0.9918$ & $-0.9755$ & 0.6335 & 0.6614 & 1.3021 & 2.6758 \\ \hline
7 & $-1.0083$ & $-1.0009$ & $-0.9991$ & $-0.9918$ & 0.6478 & 0.6615 & 1.3020 & 2.6485 \\ \hline
10 & $-1.0028$ & $-1.0000$ & $-1.0000$ & $-0.9973$ & 0.6526 & 0.6615 & 1.3020 & 2.6411 \\ \hline
\end{tabular}
\end{table}

\begin{table} 
\centering
\caption{Computed eigenvalues of $\mathcal{P}^{-1} \mathcal{A}$ and bounds, for full observation problem with $h = 2^{-5}$, $\beta = 10^{-4}$, and a range of Chebyshev semi-iterations $\ell$.}\label{Table3}
\begin{tabular}{|c||c|c||c|c||c|c||c|c|}
\hline
$\ell$ & $\text{Bound}_l^-$ & $\rho_l^-$ & $\rho_u^-$ & $\text{Bound}_u^-$ & $\text{Bound}_l^+$ & $\rho_l^+$ & $\rho_u^+$ & $\text{Bound}_u^+$ \\ \hline \hline
1 & $-8.0820$ & $-1.3891$ & $-0.4967$ & $-0.0979$ & 0.0332 & 0.3516 & 2.4877 & 4.9041 \\ \hline
2 & $-2.2751$ & $-1.2130$ & $-0.7619$ & $-0.4926$ & 0.2360 & 0.6222 & 1.6142 & 3.8810 \\ \hline
3 & $-1.2767$ & $-1.0740$ & $-0.9265$ & $-0.7966$ & 0.4729 & 0.6567 & 1.3509 & 3.0692 \\ \hline
4 & $-1.0796$ & $-1.0247$ & $-0.9750$ & $-0.9280$ & 0.5880 & 0.6584 & 1.3297 & 2.7992 \\ \hline
5 & $-1.0253$ & $-1.0082$ & $-0.9918$ & $-0.9755$ & 0.6297 & 0.6584 & 1.3281 & 2.7117 \\ \hline
7 & $-1.0083$ & $-1.0009$ & $-0.9991$ & $-0.9918$ & 0.6439 & 0.6585 & 1.3279 & 2.6807 \\ \hline
10 & $-1.0028$ & $-1.0000$ & $-1.0000$ & $-0.9973$ & 0.6487 & 0.6585 & 1.3279 & 2.6731 \\ \hline
\end{tabular}
\end{table}

\begin{table} 
\centering
\caption{Computed eigenvalues of $\mathcal{P}^{-1} \mathcal{A}$ and bounds, for full observation problem with $h = 2^{-6}$, $\beta = 10^{-4}$, and a range of Chebyshev semi-iterations $\ell$.}\label{Table4}
\begin{tabular}{|c||c|c||c|c||c|c||c|c|}
\hline
$\ell$ & $\text{Bound}_l^-$ & $\rho_l^-$ & $\rho_u^-$ & $\text{Bound}_u^-$ & $\text{Bound}_l^+$ & $\rho_l^+$ & $\rho_u^+$ & $\text{Bound}_u^+$ \\ \hline \hline
1 & $-8.5313$ & $-1.4016$ & $-0.4970$ & $-0.0979$ & 0.0331 & 0.3518 & 2.8363 & 5.3130 \\ \hline
2 & $-2.2817$ & $-1.2133$ & $-0.7619$ & $-0.4926$ & 0.2354 & 0.6117 & 1.6341 & 3.9321 \\ \hline
3 & $-1.2769$ & $-1.0740$ & $-0.9267$ & $-0.7966$ & 0.4716 & 0.6558 & 1.3626 & 3.0878 \\ \hline
4 & $-1.0796$ & $-1.0247$ & $-0.9750$ & $-0.9280$ & 0.5864 & 0.6585 & 1.3387 & 2.8115 \\ \hline
5 & $-1.0253$ & $-1.0082$ & $-0.9918$ & $-0.9755$ & 0.6280 & 0.6588 & 1.3370 & 2.7222 \\ \hline
7 & $-1.0083$ & $-1.0009$ & $-0.9991$ & $-0.9918$ & 0.6421 & 0.6588 & 1.3368 & 2.6905 \\ \hline
10 & $-1.0028$ & $-1.0000$ & $-1.0000$ & $-0.9973$ & 0.6468 & 0.6588 & 1.3368 & 2.6821 \\ \hline
\end{tabular}
\end{table}

Upon discretization of the full observation problem, and suitable re-arrangement of the resulting linear system, we obtain
\begin{equation}\label{PDECO_system}
\left(\begin{array}{ccc}
\beta M & M & O \\ M & O & L \\ O & L & M \\
\end{array}\right)\left(\begin{array}{c}
u_h \\ p_h \\ y_h \\
\end{array}\right)=\left(\begin{array}{c}
0 \\ 0 \\ \widehat{y}_h \\
\end{array}\right),
\end{equation}
where $y_h$, $u_h$, and $p_h$ denote the discretized state, control, and \emph{adjoint variables}, and $\widehat{y}_h$ arises from the discretized desired state, which in this example is a Gaussian function, that is $\widehat{y} = \text{exp}(-50((x_1-\frac{1}{2})^2+(x_2-\frac{1}{2})^2))$, with $x_1$ and $x_2$ denoting the spatial coordinates. The matrix $M$ denotes a finite element mass matrix, and $L$ the sum of a stiffness matrix and a mass matrix.

Labelling \eqref{PDECO_system} as a double saddle-point system, we have
\begin{equation*}
A = \beta M, \quad S = \frac{1}{\beta} M, \quad X = \beta L M^{-1} L + M.
\end{equation*}
Within our numerical tests, we approximate $A$ and $S$ using a number of iterations of Chebyshev semi-iteration (see \cite{GVI,GVII,WathenRees}) with Jacobi splitting to approximate the action of $M^{-1}$ on a vector. This is known to be an optimal method for the matrices under consideration \cite{WathenRees}. In particular, we are interested in the impact of varying the number of inner iterations. To approximate $X$, we take
\begin{equation*}
\widehat{X} = \frac{3}{4} \big( \sqrt{\beta} \, L + M \big)_{\text{AMG}} \, M^{-1} \, \big( \sqrt{\beta} \, L + M \big)_{\text{AMG}},
\end{equation*}
where $(\cdot)_{\text{AMG}}$ denotes the application of 2 V-cycles of the \texttt{HSL\_MI20} algebraic multigrid solver \cite{HSL_MI20,HSL_MI20_code} to a given matrix, with 2 symmetric Gauss--Seidel iterations serving as a pre- and post-smoother. The constant $\frac{3}{4}$ within $\widehat{X}$ is included to ensure the eigenvalues of the preconditioned matrix $X$ (if the $\sqrt{\beta} \, L + M$ matrices were approximated exactly) are contained within the interval $[\frac{2}{3}, \frac{4}{3}]$, that is symmetrically distributed about 1 (see \cite{PW12} for the result without this factor).

In Table \ref{Table1} we show the extremal negative and positive eigenvalues ($\rho_l^-$, $\rho_u^-$, $\rho_l^+$, $\rho_u^+$) of the preconditioned linear system, for the full observation problem with mesh parameter $h = 2^{-4}$ and $\beta = 10^{-2}$, with different numbers of Chebyshev semi-iterations $\ell$ applied to approximate $A^{-1}$ and $S^{-1}$. In Tables \ref{Table2}, \ref{Table3}, and \ref{Table4}, we present these results for $h = 2^{-4}$, $h = 2^{-5}$, and $h = 2^{-6}$, respectively, with $\beta = 10^{-4}$. We also provide the analytic bounds ($\text{Bound}_l^-$, $\text{Bound}_u^-$, $\text{Bound}_l^+$, $\text{Bound}_u^+$), which are obtained from the methodology of this paper. To arrive at these bounds, we make use of some theoretical properties of the matrices involved. For instance, given $\ell$ iterations of Chebyshev semi-iteration, we may bound $\gamma_{\min}^A$, $\gamma_{\max}^A$, $\gamma_{\min}^R$, and $\gamma_{\max}^R$ as follows:
\begin{equation*}
[\gamma_{\min}^A,\gamma_{\max}^A] \in [1-\eta, 1+\eta], \quad [\gamma_{\min}^R,\gamma_{\max}^R] \in [(1-\eta)^2, (1+\eta)^2], \qquad \text{where }~\eta = T_{\ell} \left( \frac{\lambda_{\max}^M-\lambda_{\min}^M}{\lambda_{\max}^M+\lambda_{\min}^M} \right),
\end{equation*}
with $T_{\ell}$ the $\ell$-th Chebyshev polynomial, and $\lambda_{\min}^M$ and $\lambda_{\max}^M$ denoting the minimum and maximum eigenvalues of the mass matrix preconditioned by its diagonal. 
To support the bounds for $\gamma_R$, we simply observe that 
here $\widehat{S}^{-1} \widetilde{S}$ is similar to $S_{\text{prec}} = \widehat{S}^{-1/2} \widetilde{S} \widehat{S}^{-1/2} = ( \widehat{M}^{-1/2} M \widehat{M}^{-1/2} ) ( \widehat{M}^{-1/2} M \widehat{M}^{-1/2} ) = A_{\text{prec}}^2$. 
For our tests, we use P1 elements in two dimensions, for which these eigenvalues are contained in $[\frac{1}{2}, 2]$ (see \cite{Wathen87}). Clearly, we have that $\gamma_{\min}^E \geq 0$, and simple analysis of relevant Rayleigh quotients also informs us that $\gamma_{\max}^E \leq \frac{4}{3} \gamma_{\max}^{\text{AMG}}$, $\gamma_{\min}^X \geq \frac{2}{3} \gamma_{\min}^{\text{AMG}} \gamma_{\min}^A$, and $\gamma_{\max}^X \leq \frac{4}{3} \gamma_{\max}^{\text{AMG}} \gamma_{\max}^A$. Here, $\gamma_{\min}^{\text{AMG}}$ and $\gamma_{\max}^{\text{AMG}}$ denote the minimum and maximum eigenvalues of $[ ( \sqrt{\beta} \, L + M )_{\text{AMG}} \, M^{-1} \, ( \sqrt{\beta} \, L + M )_{\text{AMG}} ]^{-1} [ ( \sqrt{\beta} \, L + M ) \, M^{-1} \, ( \sqrt{\beta} \, L + M ) ]$, that is they measure the impact of the algebraic multigrid routines on the effectiveness of the approximation of $X$. We explicitly compute $\gamma_{\min}^K$ and $\gamma_{\max}^K$ in our code, with a view to obtaining descriptive eigenvalue bounds.

As predicted by our theory, effective approximation of $A$, $S$, and $X$ leads to a potent preconditioner $\mathcal{P}$. Indeed, qualitatively speaking, our theoretical bounds capture very well where the influence of inexactness in the approximations manifests itself. As one would expect, our bounds are mesh- and $\beta$-robust, and tighten as the approximations for $A$ and $S$ improve with increasing $\ell$. We highlight that, as we have made use of theoretical, rather than exact, values for $\gamma_{\min}^A$, $\gamma_{\max}^A$, $\gamma_{\min}^R$, $\gamma_{\max}^R$, $\gamma_{\min}^E$, $\gamma_{\max}^E$, $\gamma_{\min}^X$, and $\gamma_{\max}^X$, as in practice we would like to estimate convergence based on \emph{a priori} knowledge of the problem, we do not expect the bounds to be very tight, especially for low values of $\ell$, and indeed we do not observe this to be the case in practice.

%

In Table \ref{Table5} we present the numbers of {\scshape Minres} iterations required to solve the system \eqref{PDECO_system} to a tolerance of $10^{-10}$, for each problem setup described above, as well as those for the partial observation case described below. We note that for $h = 2^{-4}$, $h = 2^{-5}$, and $h = 2^{-6}$, the linear systems have dimensions of (respectively) 867, 3267, and 12675, with these moderate dimensions being chosen simply so that eigensolvers may reasonably be applied. As expected, given the optimality of each of our approximations for $A$, $S$, and $X$, the iteration numbers are robust with respect to $h$ and $\beta$, and decrease as $\ell$ is increased.

\begin{table} 
\centering
\caption{Number of {\scshape Minres} iterations required to solve different full observation and boundary observation problems, for a range of Chebyshev semi-iterations $\ell$.}\label{Table5}
\begin{tabular}{|c||c|c|c|c||c|c|c|c|}
\hline
 & \multicolumn{4}{c||}{Full observation problems} & \multicolumn{4}{c|}{Boundary observation problems} \\ \cline{2-9}
$h$ & $2^{-4}$ & $2^{-4}$ & $2^{-5}$ & $2^{-6}$ & $2^{-4}$ & $2^{-4}$ & $2^{-5}$ & $2^{-6}$ \\ \hline
$~\ell \backslash \beta~$ & $~10^{-2}~$ & $~10^{-4}~$ & $~10^{-4}~$ & $~10^{-4}~$ & $~10^{-1}~$ & $~10^{-3}~$ & $~10^{-3}~$ & $~10^{-3}~$ \\ \hline \hline
1 & 94 & 85 & 96 & 109 & 107 & 136 & 152 & 166 \\ \hline
2 & 41 & 41 & 46 & 47 & 40 & 63 & 65 & 73 \\ \hline
3 & 26 & 29 & 30 & 31 & 24 & 42 & 45 & 45 \\ \hline
4 & 21 & 25 & 26 & 26 & 18 & 36 & 39 & 40 \\ \hline
5 & 18 & 22 & 23 & 23 & 14 & 30 & 34 & 34 \\ \hline
7 & 17 & 19 & 20 & 20 & 13 & 28 & 29 & 32 \\ \hline
10 & 14 & 18 & 19 & 19 & 12 & 25 & 28 & 28 \\ \hline
\end{tabular}
\end{table}

\subsection{Boundary Observation Problem}

Discretization of the boundary observation problem leads to the system
\begin{equation*}
\left(\begin{array}{ccc}
\beta M & M & O \\ M & O & L \\ O & L & M_{\partial\Omega} \\
\end{array}\right)\left(\begin{array}{c}
u_h \\ p_h \\ y_h \\
\end{array}\right)=\left(\begin{array}{c}
0 \\ 0 \\ \widehat{y}_{h,\partial\Omega} \\
\end{array}\right),
\end{equation*}
where $M_{\partial\Omega}$ denotes a mass matrix defined on the boundary $\partial\Omega$. The desired state $\widehat{y}$, defined only on $\partial\Omega$, is obtained by solving the forward PDE with `true' control $4x_1(1-x_1)+x_2$, a problem considered in \cite{MNN}.

In Table \ref{Table6} we show the extremal negative and positive eigenvalues of the preconditioned linear system, as well as corresponding bounds, for the boundary observation problem with $h = 2^{-4}$ and $\beta = 10^{-1}$, with different numbers of Chebyshev semi-iterations $\ell$ applied to approximate $A^{-1}$ and $S^{-1}$. In Tables \ref{Table7}, \ref{Table8}, and \ref{Table9}, we present these results for $h = 2^{-4}$, $h = 2^{-5}$, and $h = 2^{-6}$ respectively, with $\beta = 10^{-3}$. When establishing the analytic bounds, we may take the same values of $\gamma_{\min}^A$, $\gamma_{\max}^A$, $\gamma_{\min}^R$, and $\gamma_{\max}^R$ as for the full observation problem, as $S$ is the same. For the approximation of $X$, we take for this problem
\begin{equation*}
X = \beta L M^{-1} L + M_{\partial\Omega}, \qquad \widehat{X} = \big( \sqrt{\beta} \, L \big)_{\text{AMG}} \, M^{-1} \, \big( \sqrt{\beta} \, L \big)_{\text{AMG}},
\end{equation*}
noting that, unlike for the full observation problem, we do not have an optimal approximation of $X$ available, due to the highly rank-deficient term $M_{\partial\Omega}$. This does give us the opportunity to test the effect  of a relatively poor approximation of $X$, compared to those for $A$ and $S$, on the quality of our analytic bounds, due to values of $\gamma_{\max}^E$ and $\gamma_{\max}^X$ which depend on $\frac{1}{\beta}$. For this reason, we explicitly compute these values in our code. We do make use of further analytic properties: we have that $\gamma_{\min}^E = 0$, and Rayleigh quotient analysis gives us that $\gamma_{\min}^X \geq \gamma_{\min}^{\text{AMG}} \gamma_{\min}^A$, $\gamma_{\min}^K \geq \gamma_{\min}^{\text{AMG}} \gamma_{\min}^A$, and $\gamma_{\max}^K \leq \gamma_{\max}^{\text{AMG}} \gamma_{\max}^A$, with $\gamma_{\min}^{\text{AMG}}$ and $\gamma_{\max}^{\text{AMG}}$ now denoting the minimum and maximum eigenvalues of $[ ( \sqrt{\beta} \, L )_{\text{AMG}} \, M^{-1} \, ( \sqrt{\beta} \, L )_{\text{AMG}} ]^{-1} [ \beta L M^{-1} L ]$, due to $\widehat{X}$ having a different structure for the boundary observation problem.

\begin{table} [h!]
\centering
\caption{Computed eigenvalues of $\mathcal{P}^{-1} \mathcal{A}$ and bounds, for boundary observation problem with $h = 2^{-4}$, $\beta = 10^{-1}$, and a range of Chebyshev semi-iterations $\ell$.}\label{Table6}
\begin{tabular}{|c||c|c||c|c||c|c||c|c|}
\hline
$\ell$ & $\text{Bound}_l^-$ & $\rho_l^-$ & $\rho_u^-$ & $\text{Bound}_u^-$ & $\text{Bound}_l^+$ & $\rho_l^+$ & $\rho_u^+$ & $\text{Bound}_u^+$ \\ \hline \hline
1 & $-7.9780$ & $-1.3984$ & $-0.4812$ & $-0.0979$ & 0.0454 & 0.3374 & 41.494 & 43.068 \\ \hline
2 & $-2.1985$ & $-1.2103$ & $-0.7630$ & $-0.4926$ & 0.3320 & 0.7003 & 40.880 & 41.789 \\ \hline
3 & $-1.2718$ & $-1.0736$ & $-0.9263$ & $-0.7966$ & 0.6934 & 0.8869 & 40.896 & 41.211 \\ \hline
4 & $-1.0792$ & $-1.0247$ & $-0.9750$ & $-0.9280$ & 0.8735 & 0.9606 & 40.897 & 41.012 \\ \hline
5 & $-1.0252$ & $-1.0082$ & $-0.9918$ & $-0.9755$ & 0.9393 & 0.9765 & 40.897 & 40.945 \\ \hline
7 & $-1.0083$ & $-1.0009$ & $-0.9991$ & $-0.9918$ & 0.9614 & 0.9783 & 40.897 & 40.923 \\ \hline
10 & $-1.0028$ & $-1.0000$ & $-1.0000$ & $-0.9973$ & 0.9686 & 0.9783 & 40.897 & 40.915 \\ \hline
\end{tabular}
\end{table}

\begin{table} [h!]
\centering
\caption{Computed eigenvalues of $\mathcal{P}^{-1} \mathcal{A}$ and bounds, for boundary observation problem with $h = 2^{-4}$, $\beta = 10^{-3}$, and a range of Chebyshev semi-iterations $\ell$.}\label{Table7}
\begin{tabular}{|c||c|c||c|c||c|c||c|c|}
\hline
$\ell$ & $\text{Bound}_l^-$ & $\rho_l^-$ & $\rho_u^-$ & $\text{Bound}_u^-$ & $\text{Bound}_l^+$ & $\rho_l^+$ & $\rho_u^+$ & $\text{Bound}_u^+$ \\ \hline \hline
1 & $-7.9780$ & $-1.3985$ & $-0.4811$ & $-0.0979$ & 0.0454 & 0.3375 & 3991.6 & 3993.2 \\ \hline
2 & $-2.1985$ & $-1.2102$ & $-0.7619$ & $-0.4926$ & 0.3320 & 0.7004 & 3991.0 & 3991.9 \\ \hline
3 & $-1.2719$ & $-1.0736$ & $-0.9262$ & $-0.7966$ & 0.6934 & 0.8874 & 3991.0 & 3991.4 \\ \hline
4 & $-1.0792$ & $-1.0245$ & $-0.9750$ & $-0.9280$ & 0.8735 & 0.9607 & 3991.0 & 3991.2 \\ \hline
5 & $-1.0252$ & $-1.0082$ & $-0.9918$ & $-0.9755$ & 0.9393 & 0.9765 & 3991.0 & 3991.1 \\ \hline
7 & $-1.0083$ & $-1.0009$ & $-0.9991$ & $-0.9918$ & 0.9614 & 0.9787 & 3991.0 & 3991.1 \\ \hline
10 & $-1.0028$ & $-1.0000$ & $-1.0000$ & $-0.9973$ & 0.9686 & 0.9788 & 3991.0 & 3991.1 \\ \hline
\end{tabular}
\end{table}

\begin{table} [h!]
\centering
\caption{Computed eigenvalues of $\mathcal{P}^{-1} \mathcal{A}$ and bounds, for boundary observation problem with $h = 2^{-5}$, $\beta = 10^{-3}$, and a range of Chebyshev semi-iterations $\ell$.}\label{Table8}
\begin{tabular}{|c||c|c||c|c||c|c||c|c|}
\hline
$\ell$ & $\text{Bound}_l^-$ & $\rho_l^-$ & $\rho_u^-$ & $\text{Bound}_u^-$ & $\text{Bound}_l^+$ & $\rho_l^+$ & $\rho_u^+$ & $\text{Bound}_u^+$ \\ \hline \hline
1 & $-7.9786$ & $-1.4005$ & $-0.4812$ & $-0.0979$ & 0.0451 & 0.3375 & 3984.4 & 3986.0 \\ \hline
2 & $-2.1992$ & $-1.2103$ & $-0.7619$ & $-0.4926$ & 0.3399 & 0.7003 & 3983.8 & 3984.7 \\ \hline
3 & $-1.2719$ & $-1.0737$ & $-0.9263$ & $-0.7966$ & 0.6884 & 0.8866 & 3983.8 & 3984.1 \\ \hline
4 & $-1.0792$ & $-1.0245$ & $-0.9750$ & $-0.9280$ & 0.8670 & 0.9604 & 3983.8 & 3983.9 \\ \hline
5 & $-1.0252$ & $-1.0082$ & $-0.9918$ & $-0.9755$ & 0.9320 & 0.9733 & 3983.8 & 3983.9 \\ \hline
7 & $-1.0083$ & $-1.0009$ & $-0.9991$ & $-0.9918$ & 0.9538 & 0.9740 & 3983.8 & 3983.8 \\ \hline
10 & $-1.0028$ & $-1.0000$ & $-1.0000$ & $-0.9973$ & 0.9610 & 0.9740 & 3983.8 & 3983.8 \\ \hline
\end{tabular}
\end{table}

\begin{table} [h!]
\centering
\caption{Computed eigenvalues of $\mathcal{P}^{-1} \mathcal{A}$ and bounds, for boundary observation problem with $h = 2^{-6}$, $\beta = 10^{-3}$, and a range of Chebyshev semi-iterations $\ell$.}\label{Table9}
\begin{tabular}{|c||c|c||c|c||c|c||c|c|}
\hline
$\ell$ & $\text{Bound}_l^-$ & $\rho_l^-$ & $\rho_u^-$ & $\text{Bound}_u^-$ & $\text{Bound}_l^+$ & $\rho_l^+$ & $\rho_u^+$ & $\text{Bound}_u^+$ \\ \hline \hline
1 & $-7.9789$ & $-1.4013$ & $-0.4812$ & $-0.0979$ & 0.0450 & 0.3374 & 3969.3 & 3970.9 \\ \hline
2 & $-2.1994$ & $-1.2103$ & $-0.7619$ & $-0.4926$ & 0.3290 & 0.7003 & 3968.7 & 3969.6 \\ \hline
3 & $-1.2719$ & $-1.0737$ & $-0.9263$ & $-0.7966$ & 0.6862 & 0.8861 & 3968.7 & 3969.0 \\ \hline
4 & $-1.0792$ & $-1.0247$ & $-0.9750$ & $-0.9280$ & 0.8641 & 0.9549 & 3968.7 & 3968.8 \\ \hline
5 & $-1.0252$ & $-1.0082$ & $-0.9918$ & $-0.9755$ & 0.9288 & 0.9687 & 3968.7 & 3968.8 \\ \hline
7 & $-1.0083$ & $-1.0009$ & $-0.9991$ & $-0.9918$ & 0.9505 & 0.9697 & 3968.7 & 3968.7 \\ \hline
10 & $-1.0028$ & $-1.0000$ & $-1.0000$ & $-0.9973$ & 0.9577 & 0.9697 & 3968.7 & 3968.7 \\ \hline
\end{tabular}
\end{table}

We find that our analysis captures the behaviour of the outlier eigenvalues very well, specifically the largest positive eigenvalues. Interestingly, we observe that the large values of $\gamma_{\max}^E$ and $\gamma_{\max}^X$ for smaller $\beta$ only influence the largest positive eigenvalue in a noticeable way; all other (analytic and computed) bounds remain very similar. As for the full observation problem, the remaining eigenvalue bounds are descriptive as to the overall behaviour, if not necessarily tight for low numbers of Chebyshev semi-iterations. Our bounds exhibit the qualities one would expect, that is they are mesh-robust and tighten as the approximations $\widehat{A}$ and $\widehat{S}$ improve (as $\ell$ increases), so may be used as reliable indicators as to the convergence of the double saddle-point systems involved.

\section{Concluding Remarks} \label{sec:conclusion}

We have carried out a detailed spectral analysis of a family of double saddle-point
linear systems, preconditioned by the symmetric positive definite preconditioner proposed in \cite{pearson2023symmetric} within the framework of multiple saddle-point linear systems.
By means of a constrained optimization procedure, we were able to devise tight bounds for both the negative and positive intervals
containing the eigenvalues of the preconditioned matrix. Numerical results for synthetic test problems confirm the closeness
of the bounds to the endpoints of the intervals containing the spectrum.

We have illustrated the performance of this preconditioner on two realistic PDE-constrained optimization problems. Careful
selection of the block approximations $\widehat A$, $\widehat S$, and $\widehat X$ provide potent preconditioners, which we demonstrated numerically and compared with the analytic bounds. 
Our results reveal the developed bounds to be very descriptive and useful for predicting the convergence
of the MINRES iterative solver. Looseness of some bounds are observed only when the $(1,1)$ block is poorly approximated,
which also reflects on poor approximations of both the Schur complements $S$ and $X$.
Finally, we observe that
accurate approximations of $A$ and $S$ often seem to be more influential than that of $X$. The eigenvalue bounds depend on successive approximations,
meaning that inaccuracy in early blocks can ``propagate'' through the preconditioner and have a larger impact on the performance of the solver. 
In general, one should therefore permute linear systems, if possible, such that blocks which is easier to approximate appear first.

%

\subsection*{Acknowledgements}
LB and AM gratefully acknowledge the support of the INdAM-GNCS Project CUP$\_$E53C22001930001. The work of AM was carried out within the PNRR research activities of the consortium iNEST (Interconnected North-Est Innovation Ecosystem) funded by the European Union Next-GenerationEU (Piano Nazionale di Ripresa e Resilienza (PNRR) –- Missione 4 Componente 2, Investimento 1.5 -- D.D. 1058 23/06/2022, ECS$\_$00000043). 
This manuscript reflects only the Authors' views and opinions, neither the European Union nor the European Commission can be considered responsible for them. 
AP gratefully acknowledges financial support by the German Ministry of Education and Research under grant 05M22MCA.

\subsection*{Additional statements} The data that support the findings of this study are available from the corresponding author upon reasonable request. The authors have no financial or proprietary interests in any material discussed in this article.

\end{document}